\documentclass[a4paper,12pt]{article} 		
\usepackage[italian,english]{babel}			
\usepackage{amsmath,amssymb,amsthm}
\usepackage{amsfonts} 
\usepackage{graphicx} 
\usepackage{braket}
\usepackage{times}				
\usepackage{ifthen}
\usepackage{xspace}
\usepackage{fancyhdr}
\usepackage{mathrsfs}
\usepackage{enumerate}
\usepackage{verbatim}
\usepackage{hyperref}
\usepackage{mathrsfs}
\usepackage[applemac]{inputenc}

\usepackage{fancyhdr}

\usepackage{fancyhdr}

\newtheorem{theorem}{Theorem}[section]
\newtheorem{lemma}[theorem]{Lemma}
\newtheorem{definition}[theorem]{Definition}
\newtheorem{proposition}[theorem]{Proposition}
\newtheorem{remark}[theorem]{Remark}
\newtheorem{corollary}[theorem]{Corollary}
\newtheorem{example}[theorem]{Example}

\numberwithin{equation}{section}

\DeclareMathOperator*{\argmin}{argmin} 
\DeclareMathOperator*{\supp}{supp}
\DeclareMathOperator*{\esssup}{ess\ sup}

\title{\bf Long Time Behavior of First Order Mean Field Games on Euclidean Space}
\author{{\sc Piermarco Cannarsa}\footnote{%
		Dipartimento di Matematica, Universit\`a di Roma "Tor Vergata" - {\tt cannarsa@mat.uniroma2.it}}, \ 
	{\sc Wei Cheng}\footnote{%
		Department of Mathematics, Nanjing University, Nanjing 210093, China - {\tt chengwei@nju.edu.cn}}, \ 
	{\sc Cristian Mendico}\footnote{%
		Dipartimento di Matematica, Universit\`a di Roma "Tor Vergata" - {\tt cristian.mendico@gmail.com}}, \ 
	{\sc Kaizhi Wang}\footnote{%
		School of Mathematical Sciences, Shanghai Jiao Tong University, Shanghai 200240, China - {\tt kzwang@sjtu.edu.cn}}
	\\}
\date{\today}

\begin{document}
	\maketitle
	\begin{abstract}
		The aim of this paper is to study the long time behavior of solutions to deterministic mean field games systems on     Euclidean space. This problem was addressed on the torus ${\mathbb T}^n$ in [P. Cardaliaguet, {\it Long time average of first order mean field games and weak KAM theory}, Dyn. Games Appl. 3 (2013), 473--488], where solutions are shown to converge to the solution of a certain ergodic mean field games system on ${\mathbb T}^n$. By adapting the approach in [A. Fathi, E. Maderna, {\it Weak KAM theorem on non compact manifolds}, NoDEA Nonlinear Differential Equations Appl. 14 (2007), 1--27], we identify structural conditions on the Lagrangian, under which the corresponding ergodic system can be solved in $\mathbb{R}^{n}$. Then we show that time dependent solutions converge to the solution of such a stationary system on all compact subsets of the whole space.  \\\\\\
		\textit{Keywords:} Mean field games; weak KAM theory; long time behavior\\\\
		\textit{2010\ Mathematics\ Subject Classification:} 35A01; 35B40; 35F21
		\end{abstract}

\tableofcontents		

\section{Introduction}
In this paper we study the relationship between solutions of the first order mean field games (hereinafter referred to as MFG) system  with finite horizon 
\begin{equation}\label{lab1}
\begin{cases}
\ -\partial _{t} u^{T} + H(x, Du^{T})=F(x, m^{T}(t)) & \text{in} \quad (0,T)\times\mathbb{R}^{n}, \\ \  \partial _{t}m^{T}-\text{div}\Big(m^{T}D_{p}H(x, Du^{T})\Big)=0  & \text{in} \quad (0,T)\times\mathbb{R}^{n},  \\ \ m^{T}(0)=m_{0}, \quad u^{T}(T,x)=u^{f}(x), & x\in\mathbb{R}^{n},
\end{cases}
\end{equation}
and solutions of the ergodic first order MFG system
\begin{equation}\label{lab2}
\begin{cases}
\  H(x, D\bar u)=F(x, \bar m)+\bar\lambda & \text{in} \quad \mathbb{R}^{n}, \\ \ \text{div}\Big(\bar m D_{p}H(x, D\bar u)\Big)=0  & \text{in} \quad \mathbb{R}^{n},  \\ \int_{\mathbb{R}^{n}}{\bar m (dx)}=1,
\end{cases}
\end{equation}
where $0<T<+\infty$ and $H$ is a reversible strict Tonelli Hamiltonian on $\mathbb{R}^{n} \times \mathbb{R}^{n}$. More precisely, we will study the long time behavior of the solution of system \eqref{lab1} by showing  that it converges to a  solution of system \eqref{lab2} in some weak sense.

MFG theory was introduced independently by Lasry-Lions \cite{bib:LL1}, \cite{bib:LL2}, \cite{bib:LL3} and Huang, Malham\'e, and Caines \cite{bib:HCM1}, \cite{bib:HCM2} in order to study large population deterministic and stochastic differential games. 
 In system \eqref{lab1}, the function $u^{T}$ can be understood as the value function for a typical small player of a finite horizon optimal control problem in which the density $m^{T}$ of the other players enters as a datum. Moreover, the players density evolves in time, according to the second equation of the system, following the vector field given by the optimal feedback of each agent.

Our analysis is partially based on tools from weak KAM theory for Lagrangians defined on the tangent bundle of $\mathbb{R}^{n}$. Fathi \cite{bib:AF} proved the existence of solutions for stationary Hamilton-Jacobi equations, for Lagrangians defined on the tangent bundle of a compact smooth manifold, generalizing the existence result due to Lions, Papanicolaou and Varadhan \cite{bib:LPV}. Later, Fathi and Maderna \cite{bib:FM} extended this existence result to noncompact manifolds. Moreover, they showed that backward weak KAM solutions coincide with viscosity solutions.

When the state space is the flat torus $\mathbb{T}^n$, the asymptotic behavior as $T\to+\infty$ of solutions to the MFG system \eqref{lab1} was studied by Cardaliaguet \cite{bib:CAR}. In this paper, we remove such a compactness assumption and address the convergence problem as $T\to+\infty$ for solutions of \eqref{lab1} on the whole space $\mathbb{R}^n$.
The first step of our analysis is to prove the existence of solutions of system \eqref{lab2} as well as  the uniqueness of the corresponding critical value
 (Theorem \ref{MR1} below). A key point, here, is  
the regularity of viscosity solutions of the first equation of system \eqref{lab2} on the projected Mather set. 
Since such a set, for noncompact state spaces, might be empty (see, for instance, \cite{bib:GC}), we need to impose a certain structural assumption ((F4) below) on the mean field Lagrangian.

 
 Our second main result (Theorem  \ref{MR2} below) describes the behavior of the solution $(u^{T}, m^{T})$ of system \eqref{lab1}, as $T \to +\infty$, on compact subsets of $\mathbb{R}^{n}$. More precisely, let $(\bar\lambda,\bar u,\bar m)$ be a solution of \eqref{lab2}, where $\bar m$ is a projected Mather measure and $\bar\lambda$ denotes the Ma\~n\'e critical value of $H(x,p)-F(x,\bar m)$. See Definition \ref{ers} and Definition \ref{def4} below for definitions of  projected Mather measures and  Ma\~n\'e's  critical value for Tonelli Lagrangian systems, respectively.
Our first main result Theorem \ref{MR1} below guarantees the existence of such solutions $(\bar\lambda,\bar u,\bar m)$ and the uniqueness of Ma\~n\'e's critical value $\bar \lambda$. 
 We show that for every $R>R_1$ (see the definition of $R_1>0$ in Proposition \ref{mea}) there exists a  constant $C(R)>0$, such that for any $T\geq 1$ the unique solution $(u^{T}, m^{T})$ of  \eqref{lab1} satisfies: 
\begin{align*}
&\sup_{t \in [0,T]} \Big\| \frac{u^{T}(t, \cdot)-\bar u (\cdot)}{T} + \bar\lambda\left(1-\frac{t}{T}\right) \Big\|_{\infty, \overline{B}_{R}} \leq \frac{C(R)}{T^{\frac{1}{n+2}}}, 
\\& \frac{1}{T}\int_{0}^{T}{\big\| F(\cdot, m^{T}(s))- F(\cdot, \bar m) \big\|_{\infty, \overline{B}_{R}} ds} \leq \frac{C(R)}{T^{\frac{1}{n+2}}}.
\end{align*}

This paper is organized as follows: In Section 2, we fix the notation and recall preliminaries on measure theory and weak KAM theory.  
In Section 3, we prove the existence of solutions to the ergodic system \eqref{lab2} and we give a uniqueness criterion under a monotonicity assumption on $F$.
 In Section 4, after proving some preliminary lemmas,  we obtain the main convergence result. The Appendix contains the proof of ($ii$) of Theorem \ref{MR1} and a technical result which is used in the proof of Theorem  \ref{MR2}.

 \medskip

\section{Preliminaries}
In this section, we recall definitions and preliminary results from measure theory and weak KAM theory, which will be used later in this paper.
\subsection{Notation}
We write below a list of symbols used throughout this paper.
\begin{itemize}
	\item Denote by $\mathbb{N}$ the set of positive integers, by $\mathbb{R}^n$ the $n$-dimensional real Euclidean space,  by $\langle\cdot,\cdot\rangle$ the Euclidean scalar product, by $|\cdot|$ the usual norm in $\mathbb{R}^n$, and by $B_{R}$ the open ball with center $0$ and radius $R$.
	\item $\pi_{1}$ denotes the projection of $\mathbb{R}^{n} \times \mathbb{R}^{n}$  onto the first factor.
	\item  Let $a$, $b\in\mathbb{R}$. $a\vee b$ and $a\wedge b$ are used to stand for maximum and minimum, respectively: $a\vee b=\max\{a,b\}$ and $a\wedge b=\min\{a,b\}$. The positive part of a real function $f$ is defined by $f^+=f \vee 0$.
	\item Let $\Lambda$ be a real $n\times n$ matrix. Define the norm of $\Lambda$ by 
\[
\|\Lambda\|=\sup_{|x|=1, x\in\mathbb{R}^n}\|\Lambda x\|.
\]

\item Let $A$ be a Lebesgue-measurable subset of $\mathbb{R}^n$. Denote by $\mathcal{L}^{n}(A)$ the $n$-dimensional Lebesgue measure of $A$. Denote by $\mathbf{1}_{A}:\mathbb{R}^n\rightarrow \{0,1\}$ the characteristic function of $A$, i.e.,
\begin{align*}
\mathbf{1}_{A}(x)=
\begin{cases}
1  \ \ \ &x\in A,\\
0 &x \not\in A.
\end{cases}
\end{align*} 

\item Let $f$ be a real-valued function on $\mathbb{R}^n$. The set
\[
D^+ f(x)=\left\{p\in\mathbb{R}^n:\limsup_{y\to x}\frac{f(y)-f(x)-\langle p,y-x\rangle}{|y-x|}\leqslant 0\right\},
\]
is called the superdifferential of $f$ at $x$.
Let $u(t,x)$ be a real-valued function on $[0,T]\times \mathbb{R}^n$ for some $T>0$. The symbol $\nabla^+u(t,x)$ denotes the superdifferential of the function $x\mapsto u(t,x)$. 

\item ${\rm{Lip}}(A)$ stands for the space of Lipschitz functions on $A\subset\mathbb{R}^{n}$ and  denote by 
 \[
 {\rm{Lip}}(f)=\displaystyle{\sup_{\substack{x\neq y \\[1mm] x,y \in A}}}
 \frac{|f(x)-f(y)|}{|x-y|},
 \] 
the Lipschitz seminorm of $f\in{\rm{Lip}}(A)$. Define
\[
\|Df\|_{\infty,A}:=\esssup_{x\in A} |Df(x)|.
\] 

\item Let $A$ be a Lebesgue-measurable subset of $\mathbb{R}^{n}$. Let $1\leq p\leq \infty$. 
Denote by $L^p(A)$ the space of Lebesgue-measurable functions $f$ with $\|f\|_{p,A}<\infty$, where   
\begin{align*}
& \|f\|_{\infty, A}:=\esssup_{x \in A} |f(x)|,
\\& \|f\|_{p,A}:=\left(\int_{A}|f|^{p}dx\right)^{\frac{1}{p}}, \quad 1\leq p<\infty.
\end{align*}
Denote $\|f\|_{\infty,\mathbb{R}^n}$ by $\|f\|_{\infty}$ and $\|f\|_{p,\mathbb{R}^n}$ by $\|f\|_{p}$, for brevity.

\item Let $1\leq p\leq \infty$. The function $f$ belongs to the Sobolev space $W^{1,p}(\mathbb{R}^n)$ if $f\in L^p(\mathbb{R}^n)$ and for $i=1,\cdots,n$ the weak derivatives $\frac{\partial f}{\partial x_{i}}$ exist and belong to $L^p(\mathbb{R}^n)$. The function $f$ belongs to $W^{1,p}_{loc}(\mathbb{R}^n)$ if $f\in W^{1,p}(A)$ for each open set $A$ such that $\bar{A}$ is compact and $\bar{A}\subset \mathbb{R}^n$.

\item $C_b(\mathbb{R}^n)$ stands for the function space of bounded uniformly  continuous functions on $\mathbb{R}^n$. $C^{2}_{b}(\mathbb{R}^{n})$ stands for the space of bounded functions on $\mathbb{R}^n$ with bounded uniformly continuous first and second derivatives. 
$C^k(\mathbb{R}^{n})$ ($k\in\mathbb{N}$) stands for the function space of $k$-times continuously differentiable functions on $\mathbb{R}^n$, and $C^\infty(\mathbb{R}^{n}):=\cap_{k=0}^\infty C^k(\mathbb{R}^{n})$. 
 $C_c^\infty(\mathbb{R}^{n})$ stands for the space of functions in $C^\infty(\mathbb{R}^{n})$ with compact support. Let $a<b\in\mathbb{R}$.
  $AC([a,b];\mathbb{R}^n)$ denotes the space of absolutely continuous curves $[a,b]\to \mathbb{R}^n$.
  
  \item For $f \in C^{1}(\mathbb{R}^{n})$, the gradient vector of $f$ is denoted by $Df=(D_{x_{1}}f, ..., D_{x_{n}}f)$, where $D_{x_{i}}f=\frac{\partial f}{\partial x_{i}}$, $i=1,2,\cdots,n$.
Let $k$ be a nonnegative integer and let $\alpha=(\alpha_1,\cdots,\alpha_n)$ be a multiindex of order $k$, i.e., $k=|\alpha|=\alpha_1+\cdots +\alpha_n$ , where each component $\alpha_i$ is a nonnegative integer.   For $f \in C^{k}(\mathbb{R}^{n})$,
define $D^{\alpha}f:= D_{x_{1}}^{\alpha_{1}} \cdot\cdot\cdot D^{\alpha_{n}}_{x_{n}}f$. 
\end{itemize}

\medskip

\subsection{Measure theory}
Denote by $\mathscr{B}(\mathbb{R}^n)$ the  Borel $\sigma$-algebra on $\mathbb{R}^n$ and by $\mathcal{P}(\mathbb{R}^n)$ the space of Borel probability measures on $\mathbb{R}^n$.
The support of a measure $\mu \in \mathcal{P}(\mathbb{R}^n)$, denoted by $\supp(\mu)$, is the closed set defined by
\begin{equation*}
\supp (\mu) := \Big \{x \in \mathbb{R}^n: \mu(V_x)>0\ \text{for each open neighborhood $V_x$ of $x$}\Big\}.
\end{equation*}
We say that a sequence $\{\mu_k\}_{k\in\mathbb{N}}\subset \mathcal{P}(\mathbb{R}^n)$ is weakly-$*$ convergent to $\mu \in \mathcal{P}(\mathbb{R}^n)$, denoted by
$\mu_k \stackrel{w^*}{\longrightarrow}\mu$, 
  if
\begin{equation*}
\lim_{n\rightarrow \infty} \int_{\mathbb{R}^n} f(x)\,d\mu_n(x)=\int_{\mathbb{R}^n} f(x) \,d\mu(x), \quad  \forall f \in C_b(\mathbb{R}^n).
\end{equation*}

For $p\in[1,+\infty)$, the Wasserstein space of order $p$ is defined as
\begin{equation*}
\mathcal{P}_p(\mathbb{R}^n):=\{m\in\mathcal{P}(\mathbb{R}^n): \int_{\mathbb{R}^n} |x_0-x|^p\,dm(x) <+\infty\},
\end{equation*}
where $x_0 \in \mathbb{R}^n$ is arbitrary. 
Given any two measures $m$ and $m'$ in $\mathcal{P}_p(\mathbb{R}^n)$,  define
\[
\Pi(m,m'):=\left\{\lambda\in\mathcal{P}(\mathbb{R}^n\times \mathbb{R}^n): \lambda(A\times \mathbb{R}^n)=m(A),\ \lambda(\mathbb{R}^n\times A)=m'(A),\ \forall A\in \mathscr{B}(\mathbb{R}^n)\right\}.
\]
The Wasserstein distance of order $p$ between $m$ and $m'$ is defined by
    \begin{equation*}\label{dis1}
          d_p(m,m')=\inf_{\lambda \in\Pi(m,m')}\Big(\int_{\mathbb{R}^n\times \mathbb{R}^n}|x-y|^p\,d\lambda(x,y) \Big )^{1/p}.
    \end{equation*}
    The distance $d_1$ is also commonly called the Kantorovich-Rubinstein distance and can be characterized by a useful duality formula (see, for instance, \cite{bib:CV})  as follows
\begin{equation*}
d_1(m,m')=\sup\Big\{\int_{\mathbb{R}^n} f(x)\,dm(x)-\int_{\mathbb{R}^n} f(x)\,dm'(x) \ |\ f:\mathbb{R}^n\rightarrow\mathbb{R} \ \ \text{is}\ 1\text{-Lipschitz}\Big\}, 
\end{equation*}
for all $m$, $m'\in\mathcal{P}_1(\mathbb{R}^n)$.

We now recall that weak-$\ast$ convergence is equivalent to convergence in the metric space $(\mathcal{P}_{p}(\mathbb{R}^n), d_{p})$ (see, for instance,  \cite{bib:CV}). 
\begin{proposition}\label{cm}
Let $\{\mu_k\}_{k\in \mathbb{N}}$ be a sequence of measures in $\mathcal{P}_p(\mathbb{R}^n)$ and let $\mu$ be another element of $\mathcal{P}_p(\mathbb{R}^n)$.
Then
\begin{itemize}
	\item [(i)] if $d_p(\mu_k,\mu)\to 0$, then $\mu_k \stackrel{w^*}{\longrightarrow}\mu$,  as $k\to+\infty$;
	\item [(ii)] if $\supp(\mu_k)$ is contained in  a fixed compact subset of $\mathbb{R}^n$ for all $k\in \mathbb{N}$ and $\mu_k \stackrel{w^*}{\longrightarrow}\mu$,  as $k\to+\infty$, then $d_p(\mu_k,\mu)\to 0$, as $k\to+\infty$.
\end{itemize}
 \end{proposition}

Let $(X_1,S_1,\mu)$ be a measure space, $(X_2,S_2)$ a measurable space, and $f:X_1\to X_2$ a measurable map. The push-forward of $\mu$ through $f$ is the measure $f \sharp \mu$ on $(X_{2}, S_{2})$ defined by 
\begin{equation*}
f \sharp \mu(B):=\mu\left(f^{-1}(B)\right), \quad \forall B \in S_2.
\end{equation*} 
\noindent
The push-forward has the property that a measurable map $g: X_{2} \to \mathbb{R}$ is integrable with respect to $f \sharp \mu$ if and only if $g \circ f$ is integrable on $X_{1}$ with respect to $\mu$. In this case, we have that 
\begin{equation*}
\int_{X_1}g(f(x))\,d\mu(x)=\int_{X_2}g(y)\,df\sharp\mu(y).
\end{equation*}

\medskip
\subsection{Weak KAM theory on $\mathbb{R}^n$}
\begin{definition}[Tonelli Lagrangians]\label{def1}
	A $C^{2}$ function $L: \mathbb{R}^{n} \times \mathbb{R}^{n} \to \mathbb{R}$ is called a {\it Tonelli Lagrangian} if it satisfies the following: 
	\begin{itemize}
		\item[(i)] for each $(x,v) \in \mathbb{R}^{n} \times \mathbb{R}^{n} $, the Hessian $D^{2}_{vv}L(x,v)$ is positive definite;
		\item[(ii)] for each $A>0$ there exists $B(A) \in \mathbb{R}$ such that $$L(x,v) >A|v|+B(A), \quad  \forall (x,v) \in \mathbb{R}^{n} \times \mathbb{R}^{n};$$
		\item[(iii)] for each $R >0$,  $A(R):=\sup\Big\{L(x,v): |v| \leq R\Big\} < +\infty.$
	\end{itemize}
\end{definition}
\begin{definition}[Strict Tonelli Lagrangians]\label{def2}
A $C^{2}$ function $L: \mathbb{R}^{n} \times \mathbb{R}^{n} \to \mathbb{R}$ is called a {\it strict Tonelli Lagrangian} if there exist positive constants $C_{i}$ ($i=1,2,3$) such that, for all $(x,v) \in \mathbb{R}^{n} \times \mathbb{R}^{n}$:
\begin{itemize}
\item[(a)] $\frac{I}{C_{1}} \leq D_{vv}^{2}L(x,v) \leq C_{1} I$, where $I$ is the identity matrix;
	\item[(b)] $\|D^{2}_{vx}L(x,v)\| \leq C_{2}(1+|v|)$;
	\item[(c)] $|L(x,0)|+|D_{x}L(x,0)|+ |D_{v}L(x,0)| \leq C_{3}$.  
	 \end{itemize}
\end{definition}

\begin{remark}\label{re2.1}\em
Let $L$ be a strict Tonelli Lagrangian.  It is easy to check that there are two positive constants $\alpha$, $\beta$ depending only on $C_{i}$ ($i=1,2,3$) in Definition \ref{def2},  such that 
	\begin{itemize}
\item[($e$)]$|D_{v}L(x,v)| \leq \alpha(1+|v|)$, \quad $\forall (x,v)\in \mathbb{R}^n\times\mathbb{R}^n$;
\item[($f$)] $|D_{x}L(x,v)| \leq \alpha(1+|v|^{2})$, \quad $\forall (x,v)\in \mathbb{R}^n\times\mathbb{R}^n$;
\item[($g$)]$\frac{1}{4\beta}|v|^{2}- \alpha \leq L(x,v) \leq 4\beta |v|^{2} +\alpha$, \quad $\forall (x,v)\in \mathbb{R}^n\times\mathbb{R}^n$;
\item[($h$)]$\sup\big\{L(x,v): |v| \leq R \big\} < +\infty$, \quad $\forall R\geq 0$.
\end{itemize}
In view of (a), (g), and (h), it is clear that a strict Tonelli Lagrangian is a Tonelli Lagrangian. 
\end{remark}

\medskip

From now on to the end of this section, we always assume that $L$ is a Tonelli Lagrangian on $\mathbb{R}^n\times \mathbb{R}^n$.
\medskip

Define the Hamiltonian $H: \mathbb{R}^{n} \times \mathbb{R}^{n} \to \mathbb{R}$ associated with $L$ by 
$$H(x,p)=\sup_{v \in \mathbb{R}^{n}} \Big\{ \big\langle p,v \big\rangle -L(x,v) \Big\}, \quad  \forall (x,p) \in \mathbb{R}^{n} \times \mathbb{R}^{n}.$$
It is straightforward to check that if $L$ is a Tonelli Lagrangian (resp. a strict Tonelli Lagrangian), then $H$ defined above also satisfies ($i$), ($ii$), and ($iii$) in Definition 2.1 (resp. ($a$), ($b$), and ($c$) in Definition 2.2). Such a function $H$ is called a Tonelli Hamiltonian (resp. a strict Tonelli Hamiltonian). Moreover, if $L$ is a {\em reversible} Lagrangian, i.e., $L(x,v)=L(x,-v)$ for all $(x,v) \in \mathbb{R}^{n} \times \mathbb{R}^{n}$, then  $H(x,p)=H(x,-p)$ for all $(x,p) \in \mathbb{R}^{n} \times \mathbb{R}^{n}$.

Let us recall definitions of weak KAM solutions and viscosity solutions of the Hamilton-Jacobi equation 
\begin{align}\label{hj}
H(x,Du)=c, \quad x\in \mathbb{R}^n,	
\end{align}
where $c$ is a real constant.

\begin{definition}[Weak KAM solutions]\label{def3}
A function $u \in C(\mathbb{R}^{n})$ is called a backward (resp. forward) weak KAM solution of equation \eqref{hj} if: 
\begin{itemize}
\item[($i$)] for each continuous  piecewise $C^{1}$ curve $\gamma:[t_{1}, t_{2}] \to \mathbb{R}^{n}$, we have that $$u(\gamma(t_{2}))-u(\gamma(t_{1})) \leq \int_{t_{1}}^{t_{2}}{L(\gamma(s), \dot\gamma(s))ds}+c(t_{2}-t_{1});$$
\item[($ii$)] for each $x \in \mathbb{R}^{n}$, there exists a $C^{1}$ curve $\gamma:(-\infty,0] \to \mathbb{R}^{n}$ (resp. $\gamma:[0,+\infty) \to \mathbb{R}^{n}$) with $\gamma(0)=x$ such that 
\[
u(x)-u(\gamma(t))=\int_{t}^{0}{L(\gamma(s), \dot\gamma(s))ds}-ct, \quad  \forall t<0
\]
 (resp. $u(\gamma(t))-u(x) =\int_{0}^{t}{L(\gamma(s), \dot\gamma(s))ds}+ct, \quad \forall t>0$).
\end{itemize}
\end{definition}
\begin{remark}\em
A function $u$ on $\mathbb{R}^{n}$ is said to be dominated by $L+c$, denoted by $u \prec L+c$, if $u$ satisfies condition (i) of Definition \ref{def3}. A curve $\gamma$  is said to be $(u,L,c)$-calibrated if it satisfies condition (ii) of Definition \ref{def3}.
\end{remark}

\begin{definition}[Viscosity solutions]\label{visco}
	Let $V\subset \mathbb{R}^n$ be an open set.
	\begin{itemize}
		\item [($i$)] A function $u:V\rightarrow \mathbb{R}$ is called a viscosity subsolution of equation \eqref{hj}, if for every $C^1$ function $\varphi:V\rightarrow\mathbb{R}$ and every point $x_0\in V$ such that $u-\varphi$ has a local maximum at $x_0$, we have that 
		\[
		H(x_0,D\varphi(x_0))\leq c;
		\]
		\item [($ii$)] A function $u:V\rightarrow \mathbb{R}$ is called a viscosity supersolution of equation \eqref{hj}, if for every $C^1$ function $\psi:V\rightarrow\mathbb{R}$ and every point $y_0\in V$ such that $u-\psi$ has a local minimum at $y_0$, we have that
		\[
		H(y_0,D\psi(y_0))\geq c;
		\]
		\item [($iii$)] A function $u:V\rightarrow\mathbb{R}$ is called a viscosity solution of equation \eqref{hj} if it is both a viscosity subsolution and a viscosity supersolution.
	\end{itemize}
\end{definition}

\begin{definition}[Ma\~n\'e critical value]\label{def4}
The Ma\~n\'e critical value of  a Tonelli Hamiltonian $H$ is defined by\begin{equation*} 
c(H):=\inf \left\{c \in \mathbb{R} :\ \text{there exists a  viscosity solution} \ u \in C(\mathbb{R}^{n}) \ \text{of} \ H(x,Du) =c\right\}.\end{equation*}
\end{definition}
See \cite[Theorem 1.1]{bib:FM} for the following weak KAM theorem for noncompact state spaces.
\begin{theorem}[Weak KAM theorem]\label{wkt}
Let $H$ be a Tonelli Hamiltonian. Then, there exists a global viscosity solution of equation
\[
H(x, Du)=c(H),\quad x\in \mathbb{R}^n.
\]
\end{theorem}
In \cite{bib:FM},  viscosity solutions are shown to coincide with backward weak KAM solutions.

Observe that, as $\mathbb{R}^{n}$ can be seen as a covering of the torus $\mathbb{T}^{n}$, Ma\~n\'e's critical value can be characterized as follows~(\cite{bib:CIPP}): \begin{equation}\label{lab55} c(H)=\inf_{u \in C^{\infty}(\mathbb{R}^{n})} \sup_{ x \in \mathbb{R}^{n}} H(x, Du(x)) .\end{equation}

We conclude this section by recalling the notion of Mather set and the role such a set plays for the regularity of viscosity solutions. Let $L$ be a Tonelli Lagrangian. As is well known, the associated Euler-Lagrange equation, i.e., 
\begin{equation}\label{EL} \frac{d}{dt}D_{v}L(x, \dot x)=D_{x}L(x, \dot x),  \end{equation} 
generates a flow of diffeomorphisms $\phi_{t}^{L}: \mathbb{R}^{n} \times \mathbb{R}^{n} \to \mathbb{R}^{n} \times \mathbb{R}^{n}$, with $t \in \mathbb{R}$, defined by 
\begin{equation*}\label{lab11} \phi_{t}^{L}(x_{0},v_{0})=( x(t), \dot x(t)), \end{equation*} 
where $x: \mathbb{R} \to \mathbb{R}^{n}$ is the maximal solution of \eqref{EL} with initial conditions $x(0)=x_{0}, \ \dot x(0)=v_{0}$. It should be noted that, for any Tonelli Lagrangian, the flow $\phi_{t}^{L}$ is complete~(\cite{bib:FM}). 

 We recall that a Borel probability measure $\mu$ on $\mathbb{R}^{n} \times \mathbb{R}^{n}$ is called $\phi_{t}^{L}$-invariant, if $$\mu(B)=\mu(\phi_{t}^{L}(B)), \quad  \forall t \in \mathbb{R}, \quad  \forall B \in \mathscr{B}(\mathbb{R}^{n} \times \mathbb{R}^{n}),$$ or, equivalently,  $$\int_{\mathbb{R}^{n} \times \mathbb{R}^{n}} {f(\phi_{t}^{L}(x,v))\ d\mu(x,v)}=\int_{\mathbb{R}^{n} \times \mathbb{R}^{n}}{f(x,v)\ d\mu(x,v)}, \quad \forall f \in C^{\infty}_{c}(\mathbb{R}^{n} \times \mathbb{R}^{n}).$$ We denote by $\mathcal{M}_{L}$ the class of all $\phi_{t}^{L}$-invariant probability measures. 
\begin{definition}[Mather measures \cite{bib:Mat}]\label{mat}
	A probability measure $\mu \in \mathcal{M}_{L}$ is called a Mather measure for $L$, if it satisfies 
	\[  \int_{\mathbb{R}^{n} \times \mathbb{R}^{n}}{L(x,v)\ d\mu(x,v)}=\inf_{\nu \in \mathcal{M}_{L}} \int_{\mathbb{R}^{n} \times \mathbb{R}^{n}}{L(x,v)\ d\nu(x,v)}. \]
	\end{definition} 

Under the assumption {\bf (F4)} below, we deduce that the set of Mather measures is nonempty. Moreover,  in \cite{bib:FA}, it was proved that 
\begin{equation*}\label{lab44} c(H)=- \inf_{\nu \in \mathcal{M}_{L}} \int_{\mathbb{R}^{n} \times \mathbb{R}^{n}}{L(x,v)\ d\nu(x,v)}.	
\end{equation*}
Denote by $\mathcal{M}_{L}^{\ast}$ the set of Mather measures. Observe that, if $L$ (resp. $H$) is a reversible Lagrangian (resp. reversible Hamiltonian), then
\begin{equation}\label{rev}
-c(H)=\inf_{x \in \mathbb{R}^{n}} L(x,0).
\end{equation}
 The Mather set is the subset $\mathcal{M}_{0}\subset \mathbb{R}^n\times \mathbb{R}^n$   defined by 
\[
\mathcal{M}_{0} = \overline{\bigcup_{\mu \in \mathcal{M}^{\ast}_{L}} \supp(\mu)}.
\]
We call  $M_{0}=\pi_{1}(\mathcal{M}_0) \subset \mathbb{R}^{n}$ the projected Mather set. 
See \cite[Theorem 4.12.3]{bib:FA} for the following result.
\begin{theorem}\label{regularity}
	If $u$ is dominated by $L+c(H)$, then it is differentiable at every point of the projected Mather set $M_{0}$. Moreover, if $(x,v) \in \mathcal{M}_{0}$,  then $$ Du(x) =D_{v}L(x,v) $$ and the map $M_{0} \to \mathbb{R}^{n} \times \mathbb{R}^{n}$, defined by $x \mapsto (x,Du(x))$, is locally Lipschitz with a Lipschitz constant which is independent of $u$.
\end{theorem}

\medskip

\section{Ergodic MFG system: existence and uniqueness}
In this section we prove an existence and uniqueness result for \eqref{lab2}. 
\subsection{Assumptions}
From now on, we suppose that $L$ is a reversible strict Tonelli Lagrangian.
Let $F: \mathbb{R}^{n} \times \mathcal{P}_{1}(\mathbb{R}^{n}) \to \mathbb{R}$ be a function, satisfying the following  assumptions:

\begin{itemize}
	\item[\textbf{(F1)}] for every measure $m \in \mathcal{P}_{1}(\mathbb{R}^{n})$ the function $x \mapsto F(x,m)$ is of class $C^{2}_{b}(\mathbb{R}^{n})$ and	\begin{equation*}
	\sup_{m \in \mathcal{P}_{1}(\mathbb{R}^{n})} \sum_{|\alpha|\leq 2} \| D^{\alpha}F(\cdot, m)\|_{\infty} < +\infty,
	\end{equation*} 
	where $\alpha=(\alpha_1,\cdots,\alpha_n)$ and $D^{\alpha}=D^{\alpha_1}_{x_1}\cdots D^{\alpha_n}_{x_n}$;
	\item[\textbf{(F2)}] for every $x \in \mathbb{R}^{n}$ the function $m \mapsto F(x,m)$ is Lipschitz continuous and 
 $${\rm{Lip}}_{2}(F):=\displaystyle{\sup_{\substack{x\in\mathbb{R}^{n}\\ m_1,\ m_2 \in \mathcal{P}_{1}(\mathbb{R}^{n}) \\ m_1\neq m_2 } }}\frac{|F(x,m_1)-F(x,m_2)|}{d_{1}(m_1, m_2)} < +\infty;$$
	\item[\textbf{(F3)}]  there is a constant $C_F>0$ such that for every $m_{1}$, $m_{2} \in \mathcal{P}_{1}(\mathbb{R}^{n})$,
	\begin{equation*}
	\int_{\mathbb{R}^{n}}{(F(x,m_{1})-F(x, m_{2}))\ d(m_{1}-m_{2})} \geq C_F \int_{\mathbb{R}^{n}}{\left(F(x,m_{1})-F(x,m_{2})\right)^{2}\ dx},
	\end{equation*}
	and for each $x \in \mathbb{R}^{n}$,
	 \begin{equation*}
	 \int_{\mathbb{R}^{n}}{(F(x,m_{1})-F(x, m_{2}))\ d(m_{1}-m_{2})} = 0 \quad \text{if and only if}\quad    F(x,m_{1})=F(x,m_{2});
	 \end{equation*}
	 
	 \item[\textbf{(F4)}] there exist a compact set $K_{0}\subset\mathbb{R}^{n}$ and a constant $\delta_0>0$ such that, for every $m \in \mathcal{P}_{1}(\mathbb{R}^{n})$, 
	\begin{equation*}
	 \inf_{x \in \mathbb{R}^{n} \backslash K_{0}} \Big\{L(x,0)+F(x,m)\Big\}-\min_{x \in K_{0}} \Big\{L(x,0)+F(x,m) \Big\}\geq \delta_0.
	\end{equation*}
	\end{itemize}

Now we give an example where $F$ and $L$ satisfy conditions {\bf (F1)}-{\bf (F4)}.
\begin{example}\label{exa}\rm
 Let $L(x,v)=L(v)$ be a reversible strict Tonelli Lagrangian. Let 
 $$F(x,m)=f(x)g(m),$$ 
 where
\begin{itemize}
	\item  $f: \mathbb{R}^n \to \mathbb{R}$  satisfies
     \begin{itemize}
 	\item [($i$)] $f\in C^{2}_{b}({\mathbb{R}^n})$, and $\int_{\mathbb{R}^n}|f|^2\ dx<+\infty$;
 	\item [($ii$)] $\argmin_{x \in \mathbb{R}^n} f(x)$ is nonempty and bounded. 
 	  \end{itemize} 
\item $g(m)=G\big(\int_{\mathbb{R}^n}f(x)\ dm\big)$ for $m\in\mathcal{P}_1(\mathbb{R}^n)$, with $G\in C^1(\mathbb{R})$ satisfying the following:
  \begin{itemize}
 	\item [($iii$)] $G\geq \delta_{1}$, where $\delta_{1}$ is a positive constant;
 	\item [($iv$)] for each $R>0$, there is $\nu(R)>0$ such that for any $s\in[-R,R]$, 
 	\[
 	\nu(R)\leq G'(s)\leq \frac{1}{\nu(R)}.
 	\]  
   \end{itemize}
  \end{itemize}
Let $K_0\subset \mathbb{R}^n$ be a compact set such that \[
 	{\rm{int}}\ K_0\supset\argmin_{x \in \mathbb{R}^n} f(x),\]
 	where ${\rm{int}}\ K_{0}$ denotes the interior of $K_{0}$.
	
 	Then, we claim that  that assumptions {\bf (F1)}-{\bf (F4)} are fulfilled. Indeed, {\bf (F1)} and {\bf (F2)} follow, immediately,  from  ($i$), ($iv$), and the differentiability of $G$. In order to check that $F$ satisfies  {\bf (F3)},  fix $m_{1}$, $m_{2} \in \mathcal{P}_{1}(\mathbb{R}^{n})$ and observe that, since $m_{1}$, $m_{2}$ are probability measures,  
 \begin{equation*}
 \left|\int_{\mathbb{R}^{n}} f\ dm_{i}\right| \leq \| f\|_{\infty},
 \end{equation*}
for $i=1,2$. Moreover, by definition, we have that
		\begin{align*}
		&\int_{\mathbb{R}^n}(F(x,m_1)-F(x,m_2))^2\ dx&\\
		=&\left(G\left(\int_{\mathbb{R}^n}f\ dm_1\right)-G\left(\int_{\mathbb{R}^n}f\ dm_2\right)\right)^2\|f\|_{2}^{2},
		\end{align*}
		and
		\begin{align*}
		&\int_{\mathbb{R}^n}(F(x,m_1)-F(x,m_2))\ d(m_1-m_2)\\
		=&\left(G\left(\int_{\mathbb{R}^n}f\ dm_1\right)-G\left(\int_{\mathbb{R}^n}f\ dm_2\right)\right)\int_{\mathbb{R}^n}f\ d(m_1-m_2).
		\end{align*}
		If $g(m_{1})=g(m_{2})$, then  the inequality in  {\bf (F3)} is obvious. Suppose $g(m_{1})\not =g(m_{2})$. Then, $\int_{\mathbb{R}^{n}} f\ d(m_{1}-m_{2}) \not =0$ and we have that 
		\begin{align*}
			&\frac{\int_{\mathbb{R}^n}(F(x,m_1)-F(x,m_2))^2\ dx}{\int_{\mathbb{R}^n}(F(x,m_1)-F(x,m_2))\ d(m_1-m_2)}\\
			\leq &\ \|f\|_{2}^{2} \left| \frac{\left(G\left(\int_{\mathbb{R}^n}f\ dm_1\right)-G\left(\int_{\mathbb{R}^n}f\ dm_2\right)\right)}{\int_{\mathbb{R}^n}f\ d(m_1-m_2)} \right| \\
			\leq &\ \frac{\|f\|_{2}^{2}}{\nu( \|f\|_{\infty})}.
		\end{align*}	
		So far, we have checked that the inequality in {\bf (F3)} holds true.  For the necessary and sufficient condition in {\bf (F3)}, we only need to prove that $0= \int_{\mathbb{R}^{n}}{\left( F(x,m_{1})-F(x,m_{2}) \right)\ d(m_{1}-m_{2})}$ implies that $F(x,m_{1})=F(x,m_{2})$ for all $x\in \mathbb{R}^n$. Note that
		\begin{align*}
		0=& \int_{\mathbb{R}^{n}}{\left( F(x,m_{1})-F(x,m_{2}) \right)\ d(m_{1}-m_{2})} \\
		=& \left( G\left( \int_{\mathbb{R}^{n}} f\ dm_{1} \right) - G\left( \int_{\mathbb{R}^{n}} f\ dm_{2} \right) \right) \int_{\mathbb{R}^{n}} f\ d(m_{1}-m_{2}) \\
		\geq &\ \nu(\|f\|_{\infty}) \left(\int_{\mathbb{R}^{n}} f\ d(m_{1}-m_{2})\right)^2.
		\end{align*}
		Hence, $$ \int_{\mathbb{R}^{n}} f\ dm_{1}= \int_{\mathbb{R}^{n}} f\ dm_{2},$$ and thus for every $x \in \mathbb{R}^{n}$, we have that $F(x,m_{1})=F(x,m_{2}).$
		
Finally, we prove that $L$ and $F$ satisfy  {\bf (F4)}. Fix $m \in \mathcal{P}_{1}(\mathbb{R}^{n})$. 
		Since $K_{0}$ is a compact neighborhood of $\argmin_{x \in \mathbb{R}^n} f(x)$ it follows that $$\inf_{x \in \mathbb{R}^{n} \backslash K_{0}} f(x) > \min_{x \in K_{0}} f(x).$$ Then, there is a constant $\delta_{2} >0$ such that $$\inf_{x \in \mathbb{R}^{n} \backslash K_{0}} f(x)-\min_{x \in K_{0}} f(x) \geq \delta_2.$$
	 Therefore, since $g(m)
	\geq \delta_1$, we obtain that 
\begin{equation*}
\inf_{x \in \mathbb{R}^{n} \backslash K_{0}}\Big\{ L(0) +F(x,m)\Big\}-\min_{x \in K_{0}} \Big\{ L(0) +F(x,m)\Big\} \geq \delta_{1}\delta_{2}:=\delta_{0}.
\eqno{\square}
\end{equation*}

\end{example}
 
\medskip

Let $H$ be the reversible strict Tonelli Hamiltonian associated with $L$.
For any  $m\in \mathcal{P}_{1}(\mathbb{R}^{n})$, define the mean field Lagrangian and Hamiltonian associated with $m$
by 
\begin{align}
L_{m}(x,v)&:=L(x,v)+F(x,m),\,\quad (x,v)\in\mathbb{R}^n\times\mathbb{R}^n,\label{lm}\\
H_{m}(x,p)&:=H(x,p)-F(x,m),\quad (x,p)\in\mathbb{R}^n\times\mathbb{R}^n\label{hm}.
\end{align}
By assumptions {\bf (F1)} and {\bf (F2)}, it is clear that for any given $m\in \mathcal{P}_{1}(\mathbb{R}^{n})$,  $L_m$ (resp. $H_{m}$) is a strict Tonelli Lagrangian (resp. Hamiltonian).

\begin{definition}[Mean field ergodic solutions]\label{ers}
	We say that a triple $(\bar\lambda, \bar u, \bar m) \in \mathbb{R} \times C(\mathbb{R}^{n}) \times \mathcal{P}_{1}(\mathbb{R}^{n})$ is a solution of system \eqref{lab2} if 
	\begin{itemize}
		\item[($i$)] $\bar u$ is a Lipschitz continuous viscosity solution of the first equation of system \eqref{lab2};
		\item[($ii$)] $D\bar u$ exists for $\bar m-a.e. \ \ x \in \mathbb{R}^{n}$;
		\item [($iii$)] $\bar m$ is a projected Mather measure, i.e., there is a Mather measure $\eta_{\bar m}$ for $L_{\bar m}$ such that $\bar m=\pi_{1} \sharp \eta_{\bar m}$;
		\item[($iv$)] $\bar m$ satisfies the second equation of system \eqref{lab2} in the sense of distributions, that is,
		$$ \int_{\mathbb{R}^{n}}{\big\langle Df(x), D_{p}H\left(x, D\bar u(x)\right) \big\rangle\ d\bar m(x)}=0, \quad  \forall f \in C^{\infty}_{c}(\mathbb{R}^{n}). 
		$$ 
	\end{itemize}
	We denote by $\mathcal{S}$ the set of  solutions of system \eqref{lab2}.
\end{definition}
\noindent
Define the function $\lambda: \mathcal{P}_{1}(\mathbb{R}^{n}) \to \mathbb{R}$  by 
\[
 \lambda(m):=c(H_{m}).
 \]
\begin{lemma}\label{LEM1}

The function $m\mapsto \lambda(m)$ is Lipschitz continuous on $\mathcal{P}_{1}(\mathbb{R}^{n})$ with respect to the metric $d_{1}$. 
\end{lemma}

\begin{proof}
For any $m\in \mathcal{P}_{1}(\mathbb{R}^{n})$, since $L_m$ is a strict Tonelli Lagrangian,  by \eqref{lab55}  we have that
\begin{align}\label{3-100}
\lambda(m)=\inf_{u \in C^{\infty}(\mathbb{R}^{n})} \sup_{x \in \mathbb{R}^{n}} H_{m}(x, Du(x)).
\end{align}
So, the conclusion follows noting that   \eqref{3-100} and  {\bf (F2)} yield
\begin{align*}
|\lambda(m_{1})-\lambda(m_{2})| \leq \inf_{u \in C^{\infty}(\mathbb{R}^{n})} \sup_{x \in \mathbb{R}^{n}} \Big| F(x,m_{1})-F(x, m_{2}) \Big| \leq {\rm{Lip}}_{2}(F) d_{1}(m_{1},m_{2})
	\end{align*} 
for any   $m_{1}, m_{2} \in \mathcal{P}_{1}(\mathbb{R}^{n})$.
 \end{proof}
 
\medskip

\subsection{Main result 1}
We are now in a position to state and prove our first major result. 
\begin{theorem}[Existence of solutions of \eqref{lab2}]\label{MR1}
	Assume {\bf (F1)}, {\bf (F2)}, and {\bf (F4)}. 
	\begin{itemize}
		\item [($i$)] There exists at least one solution $(c(H_{\bar m}),\bar u,\bar m)$ of system $(\ref{lab2})$, i.e., $\mathcal{S}\neq\emptyset$.
		\item [($ii$)] Assume, in addition, {\bf (F3)}. Let
	$(c(H_{\bar m_{1}}), \bar u_{1}, \bar m_{1})$, $(c(H_{\bar m_{2}}), \bar u_{2}, \bar m_{2})\in \mathcal{S}$. Then,  
	\[
	F(x,\bar m_{1})= F(x,\bar m_{2}),\quad \forall x\in \mathbb{R}^{n}\quad  \text{and}\quad  c(H_{\bar m_{1}})=c(H_{\bar m_{2}}).
	\]
\end{itemize}
\end{theorem}
	
\begin{remark}\em
By ($ii$) in Theorem \ref{MR1}, it is clear that each element of $\mathcal{S}$ has the form $(\bar \lambda,\bar u,\bar m)$, where $\bar m$ is a projected Mather measure and $\bar\lambda$ denotes the common Ma\~n\'e critical value of $H_{\bar m}$.
\end{remark}	
 
\begin{proof}[Proof of Theorem \ref{MR1}]
	($i$) For any measure $m \in \mathcal{P}_{1}(\mathbb{R}^{n})$, recall that $\phi_{t}^{L_{m}}$ denotes the  Euler-Lagrange flow of $L_m$, where $L_m$ is defined in \eqref{lm}. We divide the proof of ($i$) in two steps.
	
\medskip
	
\noindent {\bf S{\footnotesize TEP 1} }: we show the existence of Mather measures for $L_m$ for each $m\in\mathcal{P}_1(\mathbb{R}^n)$.
	
\medskip
\noindent From assumption {\bf (F4)}, for any $m \in \mathcal{P}_{1}(\mathbb{R}^{n})$ there exists $ x_{m} \in K_{0}$ such that 
	\begin{equation*} 
 \inf_{x \in \mathbb{R}^n}L_{m}(x,0)= \min_{x \in K_{0}} L_{m}(x,0) =L_{m}(x_{m},0),
 \end{equation*}
 where $K_0$ is the compact set as in {\bf (F4)}. 
 Note that the constant curve $t \mapsto x_{m}$ for $ t \in \mathbb{R}$ is a solution of 
 \[ \frac{d}{dt}D_{v}L_m(x, \dot x)=D_{x}L_m(x, \dot x),  
 \] 
i.e.,  $\phi_{t}^{L_{m}}( x_{m},0)=( x_{m},0)$ for all $t \in \mathbb{R}$.  Thus, the atomic measure $\delta_{(x_{m},0)}$, supported on $(x_{m},0)$,  is a $\phi_{t}^{L_{m}}$-invariant probability measure.
 Recalling the definition of Mather measures and  $x_m$, it follows  that $\delta_{(x_{m},0)}$ is a Mather measure for $L_{m}$ (see also \cite[Proposition 4.14.3]{bib:FA}). 

Consequently, for any $m \in \mathcal{P}_{1}(\mathbb{R}^{n})$ we have that
\begin{equation*}
\mathcal{M}^{m}_{0}=\left\{ (x_{m},0) :\ L_{m}(x_{m},0)=\min_{x \in K_{0}} L_{m}(x,0) \right\},
\end{equation*}
where $\mathcal{M}^{m}_{0}$ denotes the Mather set associated with $L_{m}$.
So, for each $m \in \mathcal{P}_{1}(\mathbb{R}^{n})$, all Mather measures associated with $L_{m}$ are supported in $K_{0} \times \{0\}$.

	\medskip

	\noindent {\bf S{\footnotesize TEP 2} }: we show the existence of solutions of \eqref{lab2}.
	
\medskip  
\noindent 	From Step 1, for any $m\in\mathcal{P}_1(\mathbb{R}^n)$, there is a Mather measure $\eta_{m}$ associated with $L_m$, i.e., $\eta_{m}\in \mathcal{M}^*_{L_m}$.
	Consider the set-valued map 
	 $$\Psi: \mathcal{P}(K_{0}) \rightrightarrows \mathcal{P}(K_{0}),\quad m \mapsto \Psi(m),$$ 
	 where $$ \Psi(m):=\left\{\pi_{1} \sharp \eta_{m}:\ \eta_{m} \in \mathcal{M}_{L_m}^{\ast} \right\}.$$ 
	As is customary in MFG theory, we will apply Kakutani's theorem (see, for instance, \cite{bib:BK}) to show that there exist a fixed point $\bar{m}$ of $\Psi$.

	 Observe that the metric space $(\mathcal{P}(K_0),d_1)$ is convex and compact due to 
	 Prokhorov's theorem (see, for instance, \cite{bib:BB}).  
	 Since $\Psi$ has nonempty convex values, the only hypothesis of
	 Kakutani's theorem we need to check is that $\Psi$ has closed graph: for any pair of sequences $\{m_j\}_{j\in\mathbb{N}}\subset\mathcal{P}(K_{0})$,  $\{\mu_{j}\}_{j\in\mathbb{N}}\subset\mathcal{P}(K_{0})$ such that
	  \[
	  m_j \stackrel{w^*}{\longrightarrow}m,\quad  \mu_{j} \stackrel{w^*}{\longrightarrow}\mu,\ \text{as}\ j\to+\infty \quad \text{and} \quad \mu_j\in \Psi(m_j) \quad \text{for all j} \in\mathbb{N},
	  \]
	  we must prove that $\mu \in\Psi(m)$. Since $\mu_j\in \Psi(m_j)$ and $\mu_{j} \stackrel{w^*}{\longrightarrow}\mu$  as $j\to+\infty$, there are Mather measures $\eta_{m_j}$ and a measure $\eta\in\mathcal{P}_1(\mathbb{R}^n\times\mathbb{R}^n)$ such that 
	 \begin{align}\label{meas}
	 \mu_j=\pi_1\sharp \eta_{m_j}, \quad \eta_{m_j} \stackrel{w^*}{\longrightarrow}\eta,\  \text{as}\ j\to+\infty \quad \text{and} \quad \mu=\pi_1\sharp \eta.
	 \end{align}
	 So, it suffices to show that $\eta$ is a Mather measure for $L_m$. For this purpouse, let us consider the sequence of Ma\~n\'e's critical values $\{\lambda(m_{j})\}_{j \in \mathbb{N}}$. By \eqref{rev} and  the definition of Mather measure, we get that
	 \begin{align}\label{3-101}
	 \lambda(m_{j}) =-\int_{\mathbb{R}^{n} \times \mathbb{R}^{n}}{ L_{m_j}(x,v)\ d\eta_{m_{j}}}.
	 \end{align}
	 By \eqref{meas} and \eqref{3-101}, we deduce that $\lambda(m_{j})$ converges to some $\tilde\lambda \in \mathbb{R}$ and 
	 \[
	 \tilde\lambda =-\int_{\mathbb{R}^{n} \times \mathbb{R}^{n}}{ L_{m}(x,v)\ d\eta}.
	 \]
	  By Lemma \ref{LEM1}, we have that
	 \[
	 \tilde\lambda=\lim_{j \to \infty} \lambda(m_{j}) =\lambda(\lim_{j \to \infty} m_{j})=\lambda(m).
	 \]
	 Therefore, $\tilde\lambda$ is the Ma\~n\'e critical value of $H_m$ and $\eta$ is a Mather measure for $L_m$.
	This shows that $\Psi$ has closed graph. So, 
	 by Kakutani's  theorem, there exists $\bar m\in \mathcal{P}(K_0)$ such that $\bar m\in\Psi(\bar m)$.
	 
	 Then, by Theorem \ref{wkt}, there is a global viscosity solution $\bar{u}$ of $H_{\bar{m}}(x,Du)=c(H_{\bar m})$, where $H_{\bar m}$ is defined in \eqref{hm}. Moreover, by Theorem \ref{regularity}, $\bar u$ is differentiable $\bar m$ -a.e because $\bar m$ is supported on a  subset of the projected Mather set of $H_{\bar m}$.
	  Again by Theorem \ref{regularity} we deduce that the map $\pi_1: \supp(\eta_{\bar m}) \to \supp(\bar m)$ is one-to-one and its inverse is given by $x \mapsto (x, D_{p}H(x,D\bar u(x)))$ on $\supp(\bar m)$. 
	 
	 For any $x\in \supp(\bar m)$, let $\gamma_{t}(x)=\pi_1\circ \phi^{L_{\bar m}}_t(x,D_pH_{\bar m}(x,D\bar u(x))$. Then, we have that 
	 $$\frac{d}{dt}\gamma_{t}(x)=D_{p}H_{\bar m}\left(\gamma_{t}(x), D\bar u(\gamma_{t}(x))\right).$$ 
	 Since  $\eta_{\bar m}$ is $\phi_{t}^{L_{\bar m}}$-invariant and $\bar m$ is $\gamma_{t}$-invariant, for any function $f \in C^{\infty}(\mathbb{R}^{n})$ we get that
	 \begin{align*} 
	 0 &=\frac{d}{dt}\int_{\mathbb{R}^{n}}{f(\gamma_{t}(x))\ d\bar m(x)}= \int_{\mathbb{R}^{n}}{\big\langle Df(\gamma_{t}(x)), D_{p}H_{\bar m}(\gamma_{t}(x), D\bar u(\gamma_{t}(x))) \big\rangle\ d\bar m(x)} \\&= \int_{\mathbb{R}^{n}}{\big\langle Df(x), D_{p}H_{\bar m}(x, D\bar u(x)) \big\rangle\ d\bar m(x)}. 
	 \end{align*} 
	 Hence, $\bar m$ satisfies the second equation of system $(\ref{lab2})$ in the sense of distributions. This completes the proof of ($i$). 
	 
	 ($ii$) The proof of uniqueness, which is similar to the one in 
	  \cite{bib:CAR}, is given in the Appendix.
	 \end{proof}

\medskip

\section{MFG system with finite horizon}
This section is devoted to the second main result of this paper---the convergence result. 
Let us recall   
the MFG system with finite horizon \eqref{lab1}, i.e., 
\begin{equation*}
\begin{cases}
\ -\partial _{t} u^{T} + H(x, Du^{T})=F(x, m^{T}(t)) & \text{in} \quad (0,T)\times \mathbb{R}^{n}, \\ \  \partial _{t}m^{T}-\text{div}\Big(m^{T}D_{p}H(x, Du^{T})\Big)=0  & \text{in} \quad (0,T)\times \mathbb{R}^{n},  \\ \ m^{T}(0)=m_{0}, \quad u^{T}(T,x)=u^{f}(x), &  x\in \mathbb{R}^{n}.
\end{cases}
\end{equation*}
In this section, we will  assume {\bf (F1)}, {\bf (F2)}, {\bf (F3)},  {\bf (F4)}, and the following additional conditions.
\begin{itemize}
	\item [\textbf{(U)}] $u^f\in C^1(\mathbb{R}^{n})\cap \rm{Lip}(\mathbb{R}^{n})$ satisfies $u^{f}(x) \geq -c_{0}$ for all $x \in \mathbb{R}^{n}$ and  some constant $c_{0}\geq 0$.
	\item [\textbf{(M)}] $m_0$ is an absolutely continuous measure with respect to the Lebesgue measure  and has compact support contained in $K_{0}$, where $K_0$ is as in {\bf (F4)}. Denote by $m_{0}$ the density function of the measure $m_{0}$, i.e., $dm_{0}=m_{0}dx$.

	\item [\textbf{(F5)}] $
		\bigcap_{m \in \mathcal{P}_{1}(\mathbb{R}^{n})}{\argmin_{x \in K_0}{L_{m}( x,0)}} \not= \emptyset,
	$
where $K_0$ is as in {\bf (F4)} and $L_m$ is defined in \eqref{lm}.
\end{itemize}

\vskip.3cm

\begin{remark}\em
We observe that assumption {\bf (F5)} holds true for $L_{m}(x,v)=L(v)+f(x)g(m)$ as in Example \ref{exa}.
\end{remark} 

\begin{definition}
A pair $(u^T,m^T)\in W^{1,\infty}_{loc}([0,T] \times\mathbb{R}^n)\times L^1([0,T] \times\mathbb{R}^n)$ is said to be a solution of system \eqref{lab1} if: 
\begin{itemize}
\item [(i)] the first equation in \eqref{lab1} is satisfied in the viscosity sense;
\item [(ii)] 	the second  equation in \eqref{lab1} is satisfied in the sense of distributions. 
\end{itemize}
\end{definition}
Under assumptions {\bf (F1)}, {\bf (F2)}, {\bf (F3)}, {\bf (U)}, and {\bf (M)},
 for any given $T>0$, there exists a unique solution of \eqref{lab1} (see, for instance, \cite[Theorem 4.1 and Remark 4.2]{bib:CAR}). 
 From now on,  for any given $T>0$ we denote by $(u^{T}, m^{T})$ the unique solution of \eqref{lab1}.

  \medskip
  
 Let $R_{0}>0$ be such that $K_{0} \subset \overline{B}_{R_{0}}$, where $K_0$ is as in {\bf (F4)}.

\medskip  

Let $\ell$ be a time-dependent Tonelli Lagrangian and let $u$ be a continuous, bounded below function on $\mathbb{R}^n$. For any given $x\in \mathbb{R}^n$, and $t$, $T\in\mathbb{R}$ with $0<t<T$,  classical results (see, for instance, \cite[Theorem 6.1.2]{bib:SC}) ensure the existence of solutions of 
 the following minimization problem\begin{equation*}\label{lab4}
\inf_{\xi \in \Gamma_{t,T}(x)}\left\{\int_{t}^{T}{\ell(s,\xi(s), \dot \xi(s)) ds} + u(\xi(T))\right\},
\end{equation*}
where  
\begin{equation*}\label{gam}
\Gamma_{t,T}(x):=\{\gamma\in AC([t,T];\mathbb{R}^n): \gamma(t)=x\}.
\end{equation*}

For any given $T>0$, let $(u^T,m^T)$ be the unique solution of  \eqref{lab1}. For each $x\in\mathbb{R}^n$, each $t\in[0,T]$, consider the minimization problem 
	\begin{equation}\label{infimum}
	\inf_{\xi \in \Gamma_{t,T}(x)}\left\{\int_{t}^{T}{L_{m^T (s)}\left(\xi(s), \dot \xi(s)\right)\ ds} + u^{f}(\xi(T))\right\}.
	\end{equation}  
 Define 
\begin{equation*}
\Gamma^{*}_{t,T}(x):=\Big\{ \xi^{*} \in \Gamma_{t,T}(x) :\ \xi^{*}\ \text{is a solution of problem \eqref{infimum}} \Big\}. 
\end{equation*}
Consider the set-valued map  
\[
\Gamma^{*}_{t,T}: K_0\rightrightarrows C^1([t,T];\mathbb{R}^n), \quad x\mapsto \Gamma^{*}_{t,T}(x).
\]
 It is easy to check that  $\Gamma^{*}_{t,T}$ has closed graph with respect to the $C^{1}$-topology, which implies that $\Gamma^{*}_{t,T}$ is Borel measurable with closed values (see, for instance, \cite[Proposition 9.5]{bib:C}).   Therefore, by the measurable selection theorem (see, for instance, \cite{bib:AFS}), there exists a measurable selection of $\Gamma^{*}_{t,T}$, that is, $\gamma^*: K_0\to C^1([t,T];\mathbb{R}^n)$ such that $\gamma^*(x)\in  \Gamma^{*}_{t, T}(x)$ for all $x\in K_0$.
 For any $s\in[t,T]$, let us consider the {\em evaluation map} $ e_s:C^1([t,T];\mathbb{R}^n) \to \mathbb{R}^n$, that is,
 \[
 e_s( \gamma) = \gamma(s).
 \]
Then, we  define the {\it optimal flow} as follows:
\begin{equation*}
\phi:[t,T]\times K_0\to \mathbb{R}^n,\quad \phi(s,x)=e_s(\gamma^*(x))\quad (s\in[t,T], x\in K_0).
\end{equation*}
Moreover, from \cite[Lemma 4.15]{bib:CN}, we have that, for any $T>0$, 
\begin{equation}\label{sol}
m^{T}(s)=\phi(s, \cdot) \sharp m_{0}, \quad \forall s \in [t,T].
\end{equation} 

\medskip
\subsection{Excursion time of minimizers}
Before proving Theorem \ref{MR2} below, we derive preliminary results of interest in their own right.

\begin{theorem}[Excursion time from a compact set]\label{4-3}
	For any $R\geq R_0$  
	there is $M_R>0$ such that for any $T> 1$, any $\tilde x\in \overline{B}_R$, and any 
	      $\xi^{\ast}\in \Gamma^{*}_{0,T}(\tilde{x})$, 
	we have that
	\[
	\mathcal{L}^1\Big(\{s\in[0,T]:\ \xi^{*}(s)\in \overline{B}_{R}\}\Big)\geq T-M_R.
	\]
\end{theorem}

\begin{proof}
Define
\begin{equation*}
b=\min \big\{0,\ \inf_{\substack{(x,v) \in \mathbb{R}^{n} \times \mathbb{R}^{n} \\ m \in \mathcal{P}_{1}(\mathbb{R}^{n})}} L_{m}(x,v) \big\}.
\end{equation*}
 By assumption {\bf (F5)},
 there exists $x_{T} \in K_0$ such that 
 $$\displaystyle{\min_{x \in K_0}}\ {L_{m^T (s)}(x,0)}=L_{m^T (s)}(x_{T},0)$$ 
  for $s \in [T_{0}, T]$. Consider the curve 
$$ \xi_{0}(s):= \begin{cases} 
	\tilde x+\frac{x_{T}-\tilde x}{T_{0}} \cdot s, & \quad s\in [0, T_{0}), \\ x_{T}, & \quad s \in [T_{0}, T].
	\end{cases}
	$$
 From the minimality of $\xi^*$, we  have that
 \begin{align*} 
& \int_{0}^{T}{L_{m^T (s)}(\xi^{\ast}(s), \dot\xi^{\ast}(s))\ ds}+u^{f}(\xi^{\ast}(T)) \\
  \leq & \int_{0}^{T_{0}}{L_{m^T (s)}(\xi_{0}(s), \dot\xi_{0}(s))\ ds} + \int_{T_0}^{T}{L_{m^T (s)}(x_{T}, 0)\ ds}+u^{f}(\bar x)\\
  =&\ c(T_{0},\tilde x,m^T)+\int_{T_0}^{T}{L_{m^T (s)}(x_{T}, 0)\ ds},
  \end{align*} 
 where  $$c(T_{0},\tilde x,m^T)=\int_{0}^{T_{0}}{L_{m^T (s)}(\xi_{0}(s), \dot\xi_{0}(s))\ ds}+u^{f}(x_{T}).$$ By our assumptions, we deduce that
 $
 |c(T_{0},\tilde x,m^T)|\leq C(T_0,R),
 $
 where $C(T_0,R)>0$ depends only on $T_0$ and $R$.
 On the other hand, by the  convexity and reversibility of $L$ with respect to the  $v$, we deduce that $L(x,v)\geq L(x,0)$. Thus, we have that 
	\begin{align*} 
	&\int_{0}^{T}{L_{m^T (s)}(\xi^{\ast}(s), \dot\xi^{\ast}(s))\ ds}  \\ 
	 \geq &\ bT_0+ \int_{T_0}^{T}{L_{m^T (s)}(\xi^{\ast}(s), 0){\bf 1}_{\overline{B}_{R}}(\xi^{\ast}(s))\ ds} +  \int_{T_0}^{T}{L_{m^T (s)}(\xi^{\ast}(s), 0){\bf 1}_{\mathbb{R}^{n} \backslash \overline{B}_{R}}(\xi^{\ast}(s))\ ds}\\
	 \geq &\ bT_0+ \int_{T_0}^{T}{L_{m^T (s)}(x_{T}, 0){\bf 1}_{\overline{B}_{R}}(\xi^{\ast}(s))\ ds} +  \int_{T_0}^{T}{\inf_{x \in \mathbb{R}^{n}} L_{m^T (s)}(x, 0){\bf 1}_{\mathbb{R}^{n} \backslash \overline{B}_{R}}(\xi^{\ast}(s))\ ds}.
	 \end{align*} 
	 By combining the above inequalities, we deduce that 
	\begin{align*}
		\int_{T_0}^{T}{\Big( L_{m^T (s)}(x_{T},0)-\inf_{x \in \mathbb{R}^n \backslash \overline{B}_{R}}{L_{m^T (s)}(x,0)} \Big) \Big( 1-{\bf 1}_{\overline{B}_{R}}(\xi^{\ast}(s))\Big)\ ds} \geq bT_0-C(T_0,R).
	\end{align*}
	By assumption {\bf (F4)} and the above inequality, we get 
	\[
	\int_{T_0}^{T}{ \Big( 1-{\bf 1}_{\overline{B}_{R}}(\xi^{\ast}(s))\Big)\ ds} \leq \frac{C(T_0,R)-bT_0}{\delta_0}=:M_R
		\]
which yields the conclusion.
\end{proof}

\begin{remark}\label{re}\em
In view of the proof of Theorem \ref{4-3}, it is clear that the result still holds true if assumption {\bf (F5)} is replaced by the following
\vskip.2cm

\noindent \textbf{(F5')}
 Let $(u^T,m^T)$ be a solution of system \eqref{lab1}.
 There exists $T_0\in(0,T)$ such that 
 
	\begin{equation*}
		\bigcap_{T_0\leq s\leq T }{\argmin_{x \in K_0}{L_{m^T(s)}( x,0)}} \not= \emptyset.
	\end{equation*}
	It is notable that {\bf (F5)} is equivalent to the  assumption: for every compact subset $\mathcal{J}$ of $\mathcal{P}_1(\mathbb{R}^n)$, 
	\begin{equation}\label{assplus}
		\bigcap_{m \in \mathcal{J}}{\argmin_{x \in K_0}{L_{m}( x,0)}} \not= \emptyset,
	\end{equation}
where $K_0$ is as in {\bf (F4)} and $L_m$ is defined in \eqref{lm}. Since $\{m^{T}(s)\}_{s \in [T_0,T]}$ is a compact subset of $\mathcal{P}_{1}(\mathbb{R}^{n})$, 
	then assumption {\bf (F5')} is just an application of condition \eqref{assplus} to $\{m^{T}(s)\}_{s \in [T_0,T]}$.
\end{remark}
\begin{lemma}\label{lemma}
Let $R\geq R_0$. Then there exists a constant $\kappa(R)>0$  such that for any $T > 1$, any $0
\leq t\leq T$,  any $x \in \overline{B}_{R}$, and any $\xi^{*} \in \Gamma^{*}_{t,T}(x)$ we have that 
\begin{equation*}
\int_{t}^{(t+1) \wedge T}{ |\dot\xi^{*}(s)|^{2}\ ds} \leq \kappa(R).
\end{equation*}
\end{lemma}

\begin{proof}
We consider two cases.

\medskip
\noindent {\bf C\footnotesize{ASE} 1}:
$ t \in [0,T-1)$. 

\medskip

\noindent By assumptions {\bf (F4)} and {\bf (F5)}, for any $T \geq 1$ there exists $x_{T} \in K_{0}$ such that for any $s \in [0,T]$ we have that
\begin{equation*}
L_{m^{T}(s)}(x_{T},0)= \min_{x \in K_{0}}\ L_{m^{T}(s)}(x,0) \leq \inf_{x \in \mathbb{R}^{n} \backslash K_{0}} \ L_{m^{T}(s)}(x,0) - \delta_{0}.
\end{equation*}
So, in view of ($f$) in Remark \ref{re2.1}  and the reversibility of $L$, we get that
\begin{align}\label{si}
\begin{split}
u^{T}(t,x)
& =  \int_{t}^{T}{L_{m^{T}(s)}(\xi^{*}(s), \dot\xi^{*}(s))\ ds} + u^{f}(\xi^{*}(T)) \\
& =  \int_{t}^{t+1}{L_{m^{T}(s)}(\xi^{*}(s), \dot\xi^{*}(s))\ ds} + \int_{t+1}^{T}{L_{m^{T}(s)}(\xi^{*}(s), \dot\xi^{*}(s))\ ds}+ u^{f}(\xi^{*}(T)) \\
& \geq  \frac{1}{4\beta} \int_{t}^{t+1}{|\dot\xi^{*}(s)|^{2}\ ds} - \alpha + \int_{t+1}^{T}{L_{m^{T}(s)}(x_{T},0)\ ds} - c_{0}.
\end{split}
\end{align}
For $x \in \overline{B}_{R}$,   define
\begin{equation*}
\sigma^{x}_{T}(s)=\left((s-t)\wedge 1\right)x_{T} + \left((1-(s-t))\wedge 1)\right)x, \quad t\leq s \leq T.
\end{equation*}
Observe that $\sigma^{x}_{T}(t)=x$ and $\sigma^{x}_{T}(s)=x_{T}$ for any $s \geq t+1$.
We deduce that
\begin{align*}
u^{T}(t,x)
= &\ \int_{t}^{T}{L_{m^{T}(s)}(\xi^{*}(s), \dot\xi^{*}(s))\ ds} + u^{f}(\xi^{*}(T))\\
\leq &\ \int_{t}^{t+1}{L_{m^{T}(s)}(\sigma^{x}_{T}(s), \dot\sigma^{x}_{T}(s))\ ds} + \int_{t+1}^{T}{L_{m^{T}(s)}(x_{T},0)\ ds} + u^{f}(x_{T}).
\end{align*}
Since
$
|\sigma^{x}_{T}(s)| \leq 2R$ and $|\dot\sigma^{x}_{T}(s)| \leq 2R$ for all $s \in [t,T]$,
we have that
\begin{equation}\label{no}
u^{T}(t,x) \leq C(R) + u^{f}(x_{T}) + \int_{t+1}^{T}{L_{m^{T}(s)}(x_{T},0)\ ds},
\end{equation}
where 
$$C(R)=\displaystyle{\max_{|x|,|v| \leq 2R}}\ |L(x,v)| + \|F(\cdot, m^{T})\|_{\infty}.$$ 
Note that $x_{T}\in K_0$. Since $K_0$ is compact,  by {\bf (U)} we deduce that $u^{f}(x_{T})$ is bounded. 
So, combining \eqref{si} and \eqref{no}, we conclude that
\begin{equation*}
\int_{t}^{t+1}{|\dot\xi^{*}(s)|^{2}\ ds} \leq 4\beta\big(\alpha+C(R)+2c_{0}\big)=:\kappa_1(R).
\end{equation*}

\medskip
\noindent {\bf C\footnotesize{ASE} 2}:
$ t \in [T-1,T]$. 

\medskip 
\noindent The proof is similar to the one above. On the one hand,  one has that
\begin{align*}
u^{T}(t,x) \geq \frac{1}{4\beta} \int_{t}^{T}{|\dot\xi^{*}(s)|^{2}\ ds} -\alpha -c_{0}.
\end{align*}
On the other hand, by using the curve $\rho(s)\equiv x$ we get an upper bound of the form
\begin{align*}
u^{T}(t,x) \leq \int_{t}^{T}{L_{m^{T}(s)}(x,0)\ ds} + u^{f}(x).
\end{align*}
Therefore, combining the above inequalities we obtain the desired result.
\end{proof}

\begin{remark}\label{re400}\ 
\begin{itemize}\em
\item[(a)] By minor adaptations of the above proof one can extend the conclusion of  Lemma \ref{lemma} as follows: for every $R\geq R_0$ and $M>0$ there exists a constant $\kappa(R, M)>0$  such that for any $T > 1$, any $0
\leq t\leq T$,  any $x \in \overline{B}_{R}$, and any $\xi^{*} \in \Gamma^{*}_{t,T}(x)$ we have that 
\begin{equation*}
\int_{t}^{(t+M) \wedge T}{ |\dot\xi^{*}(s)|^{2}\ ds} \leq \kappa(R,M).
\end{equation*}
\item[(b)] The proof of Lemma \ref{lemma} shows that $\kappa(R)=C R^{2}$ for some  constant $C>0$.
\end{itemize}
\end{remark}
\medskip
\begin{corollary}\label{distance}
Let $R\geq R_0$. Then, there exists a constant $\chi(R)>0$  such that for any $T > 1$, any $0
\leq t\leq T$,  any $x \in \overline{B}_{R}$, and any $\xi^{*} \in \Gamma^{*}_{t,T}(x)$ we have that  
\begin{equation*}
\sup_{s \in [t,T]} |\xi^{*}(s)| \leq \chi(R).
\end{equation*}
\end{corollary}
\begin{proof}
Fix $x \in \overline{B}_{R}$. Let $\bar{t}\in[0,T]$ be such that $\xi^{*}(\bar{t},x)\in\mathbb{R}^{n} \backslash \overline{B}_{R}$. 
Define
\[
s_0=\sup\{s\in[0,\bar{t}\ ]: \xi^{*}(s)\in \overline{B}_{R}\}.
\]
Then, we have that
\begin{align*}
\left| \xi^{*}(\bar{t})-\xi^{*}(s_0)\right|&=\int_{s_0}^{t}{\left|\dot\xi^{*}(s)\right|\ ds} \leq\ M_{R}^{\frac{1}{2}}\left(\int_{s_0}^{t}{\left|\dot\xi^{*}(s)\right|^{2}\ ds}\right)^{\frac{1}{2}}\leq\ M_{R}^{\frac{1}{2}}\kappa(R, M_{R})^{\frac{1}{2}},
\end{align*}
where we have used H\"older's inequality and $\kappa(R, M_{R})$ is as in Remark \ref{re400}. Therefore, we get $\chi(R)=R+\left(M_{R}\kappa(R,M_{R})\right)^{\frac{1}{2}}$.
\end{proof}

\begin{proposition}[Attainable set from $K_0$]\label{mea}
 There exists a constant $R_1>0$ such that for each $T\geq 1$ the solution $(u^T,m^T)$ of system \eqref{lab1} satisfies
\begin{equation*}
\supp\left({m^{T}(t)}\right) \subset \overline{B}_{R_1}, \quad \forall t \in [0,T].
\end{equation*}
\end{proposition}
\begin{proof}
Recall that $R_0$ has been fixed so that $K_0 \subset \overline{B}_{R_0}$ and $m^{T}(t)=\phi(t,\cdot)\sharp m_{0}$. So, by assumption {\bf(M)}, Corollary \ref{distance}, and the defintion of $\phi$  
 we conclude that 
 \begin{equation*}
\supp\left({m^{T}(t)}\right) \subset \overline{B}_{R_1}, \quad \forall T>0, \quad \forall t \in [0,T],
\end{equation*}
where  $R_1:=\chi(R_{0})$. 
\end{proof}

\subsection{Uniform Lipschitz continuity}

\begin{proposition}\label{lip}
Let $R\geq R_0$. Then, there exists a  constant $L_{R}>0$ such that for all $T> 1$
\begin{equation*}
|u^{T}(t,x+h)-u^{T}(t,x)| \leq L_{R}|h|, \quad \forall x, x+h \in \overline{B}_{R}, \ \forall t\in[0,T].
\end{equation*} 
\end{proposition}

\begin{proof}
Since $(u^{T}, m^{T})$ is the solution of \eqref{lab1}, we have that 	
	\begin{align*}
	\begin{cases}
	-\partial_{t} u^{T} + H(x, Du^{T})=F(x, m^{T}) \quad \text{in} \quad [0,T] \times \mathbb{R}^{n},  \\ u^{T}(T,x)=u^{f}(x) \quad \quad\quad\quad\quad\quad\quad\,\,\,\,  \text{in} \quad  \mathbb{R}^{n},
	\end{cases}
	\end{align*}
and
\begin{equation}\label{ggg}
u^{T}(t,x)=\inf_{\xi \in \Gamma_{t,T}(x)}\left\{\int_{t}^{T}{L_{m^{T}(s)}\left(\xi(s), \dot \xi(s)\right) ds} + u^{f}(\xi(T))\right\}.
\end{equation}
Fix $R\geq R_0$ and $x \in \overline{B}_{R}$. Let $h\in \mathbb{R}^{n}$ be such that $x+h \in \overline{B}_{R}$.
We consider two cases.

\medskip
\noindent {\bf C\footnotesize{ASE} 1}:
$ t \in [0,T-1)$. 

\medskip
\noindent Let $\xi^{*}\in\Gamma^*_{t,T}(x)$.  By testing with the curve 
$$\xi_{h}(s):=\xi^{*}(s)+(t+1-s)^{+}h\quad(s \in [t,T]),$$
we obtain the upper bound
\begin{align*}
u^{T}(t,x+h)-u^{T}(t,x)
\leq &\ \int_{t}^{t+1}{\left( L_{m^{T}(s)}(\xi_{h}(s), \dot\xi_{h}(s)) - L_{m^{T}(s)}(\xi^{*}(s), \dot\xi^{*}(s))\right)\ ds} \\
\leq &\ \underbrace{\int_{t}^{t+1}{\left( L_{m^{T}(s)}(\xi_{h}(s), \dot\xi_{h}(s)) - L_{m^{T}(s)}(\xi^{*}(s), \dot\xi_{h}(s))\right)\ ds}}_{A}\\
&\ +   \underbrace{\int_{t}^{t+1}{\left( L_{m^{T}(s)}(\xi^{*}(s), \dot\xi_{h}(s)) - L_{m^{T}(s)}(\xi^{*}(s), \dot\xi^{*}(s))\right)\ ds}}_{B}.
\end{align*}
We estimate term A first. 
Recall that, for any $m \in \mathcal{P}_{1}(\mathbb{R}^{n})$, $L_{m}$ is a strict Tonelli Lagrangian. By ($e$) in Remark \ref{re2.1} and (F1), there exists a constant $\alpha_1>0$ such that 
\begin{align}\label{4-100}
|D_{x}L_{m}(x,v)| \leq \alpha_1(1+|v|^{2}),\quad \forall (x,v) \in \mathbb{R}^{n} \times \mathbb{R}^{n},\ \forall m\in \mathcal{P}_{1}(\mathbb{R}^{n}).
\end{align}
Then, by Lemma \ref{lemma}, we get 
\begin{align*}
A  \leq & \int_{t}^{t+1}\int_{0}^{1}{|h|\left|  D_{x}L_{m^{T}(s)}\big(\lambda \xi_{h}(s) +(1-\lambda)\xi^{*}(s), \dot\xi_{h}(s)\big)\right|\ d\lambda ds}\\
\leq &\ \alpha_1 |h| \int_{t}^{t+1}{ \big( 1 + |\dot\xi_{h}(s)|^{2} \big)\ ds} = \ \alpha_1 |h| \int_{t}^{t+1}(1+|\dot{\xi}^*-h|^2)\ ds
\\
\leq &\ (3+2R)\alpha_1|h| \left( 1 + \int_{t}^{t+1}{|\dot\xi^{*}(s)|^{2}\ ds} \right) \leq \kappa'(R) |h|,
\end{align*}
where $\kappa'(R)=(3+2R)\alpha_1(1+\kappa(R))$ and $\kappa(R)$ is as in Lemma \ref{lemma}. 
Similarly,  by Lemma \ref{lemma} and ($d$) in Remark \ref{re2.1} we get 
$
B \leq  \kappa''(R) |h|, 
$
where $\kappa''(R)$ is a positive constant depending only on $R$.

\medskip
\noindent {\bf C\footnotesize{ASE} 2}: 
$t \in [T-1,T]$.

\medskip
\noindent Let $\xi^{*}\in\Gamma^*_{t,T}(x)$. Define the curve $\xi(s)=\xi^{*}(s)+h$, for $s \in [t,T]$. Then, we have that
\begin{align*}
& u^{T}(t,x+h)-u^{T}(t,x)\\
\leq &\ \int_{t}^{T}{\left( L_{m^{T}(s)}(\xi(s), \dot\xi(s)) - L_{m^{T}(s)}(\xi^{*}(s), \dot\xi^{*}(s))\right)\ ds}+ {\rm{Lip}}(u^{f})|h|.
\end{align*}
 To conclude the proof, we only need to estimate  
\begin{equation*} 
 \int_{t}^{T}{\left( L_{m^{T}(s)}(\xi(s), \dot\xi(s)) - L_{m^{T}(s)}(\xi^{*}(s), \dot\xi^{*}(s))\right)\ ds}.
 \end{equation*}
Again by Lemma \ref{lemma} and \eqref{4-100} we get
\begin{align*}
&\int_{t}^{T}{\left( L_{m^{T}(s)}(\xi(s), \dot\xi(s)) - L_{m^{T}(s)}(\xi^{*}(s), \dot\xi^{*}(s))\right)\ ds}\\
\leq &\ \int_{t}^{T}\int_{0}^{1}{|h|\left| D_{x}L_{m^{T}(s)}(\lambda h+\xi^{*}(s), \dot\xi^{*}(s))\right|\ d\lambda ds}\\
\leq &\ \alpha_1 |h| \left( 1 + \int_{t}^{T}{|\dot\xi^{*}(s)|^{2}\ ds} \right) \leq \kappa'''(R) |h|,
\end{align*}
where $\kappa'''(R)=\alpha_1(1+\kappa(R))$.
So, $u^{T}(t,x+h)-u^{T}(t,x) \leq L_{R}|h|$ with $L_R=\kappa'''(R)+{\rm{Lip}}(u^{f})$. This suffices to get the conclusion.
\end{proof}

\begin{corollary}\label{vel}
Let $R\geq R_0$.  Then there exists a  constant $\chi^{\prime}(R)>0$ such that for any $T>1$, any $0\leq t\leq T$, any $x \in \overline{B}_{R}$, and any $\xi^{*} \in \Gamma^{*}_{t,T}(x)$, we have that \begin{equation*}
\sup_{s \in [t,T]}\ |\dot\xi^{*}(s)| \leq \chi^{\prime}(R).
\end{equation*}
\end{corollary}
\begin{proof}
Let $p^{*}$ be the dual arc of $\xi^{*}$, 
 that is, $p^{*}(s)=D_{v}L_{m^T(s)}(\xi^{*}(s), \dot\xi^{*}(s))$ for any $s \in [t,T]$.
Then the pair $(\xi^{*},p^{*})$ satisfies 
the maximum principle in Hamiltonian form
\begin{equation*}
\begin{cases}
\dot\xi^{*}(s)=D_{p}H_{m^T(s)}(\xi^{*}(s), p^{*}(s))
\\
\dot p^{*}(s)=-D_{x}H_{m^T(s)}(\xi^{*}(s), p^{*}(s))
\end{cases}
\quad(s \in [t,T])
\end{equation*}
Moreover, by \cite[Theorem 6.4.8]{bib:SC} we know that 
\begin{align*}
p^{*}(s)=D_{x}u^{T}(s,\xi^{*}(s)), \quad t<s\leq T.
\end{align*}

Now, observe that, in view of Corollary~\ref{distance}, $\xi^{*}(s)\in  \overline{B}_{\chi(R)}$ for all $s\in[t,T]$ and, on account of Proposition~\ref{lip}, $\{u^{T}(s, \cdot)\}_{s\in[t,T]}$ is equi-Lipschitz continuous on $\overline{B}_{\chi(R)}$. Therefore,  there exists a positive constant $c_{R}$, independent of $T$, such that $|p^*(s)| \leq c_{R}$ for every $s \in [t,T]$. 
 Consequently,
\begin{equation*}
\sup_{s \in [t,T]}\ |\dot\xi^{*}(s)|=\sup_{s \in [t,T]}|D_{p}H_{m^T(s)}(\xi^{*}(s), p^{*}(s))| \leq \alpha_2 \big( 1 + c_{R}\big):=\chi^{\prime}(R),
\end{equation*}
where the inequality follows from the fact that $D_pH_{m^T(s)}(x,p)\leq \alpha_2(1+|p|)$ for some $\alpha_2>0$ and all $(x,p)\in\mathbb{R}^n\times\mathbb{R}^n$ and  $s\in[t,T]$.
\end{proof}

\begin{remark}\em
	Owing to Corollary \ref{vel}, for any $T>1$ we have that $m^T:[0,T]\to  \mathcal{P}_{1}(\mathbb{R}^{n})$ is Lipschitz continuous.
	Indeed, by \eqref{sol} we deduce that, for all $s',\ s\in[0,T]$,
	\begin{align*}
		d_1(m^T(s'),m^T(s))\leq \int_{K_0}|\phi(s',x)-\phi(s,x)|\ dm_0\leq \chi'(R_0)|s'-s|,
	\end{align*}
	where $R_0$ is such that $K_0 \subset \overline{B}_{R_0}$.
\end{remark}

\medskip

\medskip
\subsection{Main result 2}
Before proving our main result, we show the following lemma. We recall that $R_1$ is the constant given by Proposition \ref{mea}.
\begin{lemma}\label{energy}
For any $R\geq R_{1}$ there exists a constant $ C(R)>0$, such that for any $T > 0$, and any $(\bar\lambda, \bar u, \bar m)\in \mathcal{S}$, the solution $(u^{T}, m^{T})$ of \eqref{lab1} satisfies
	\begin{equation*}
	\int_{0}^{T} \int_{\overline{B}_{R}}{\Big(F(x, m^{T}(t))- F(x, \bar m)\Big)d(m^{T}(t)-\bar m) dt} \leq \tilde C(R),
	\end{equation*}	
	where $R_1$ is as in Proposition \ref{mea}.
\end{lemma}
\begin{proof}
Fix $R\geq R_1$. Then, $K_{0} \subset \overline{B}_{R}$ and $\partial \overline{B}_{R} \cap K_{0} = \emptyset$. 
Let $\epsilon >0$ and let $\xi : \mathbb{R}^{n} \to \mathbb{R}$ be a smooth, nonnegative, symmetric kernel of integral one,  with support contained in the unit ball. Fix $(\bar\lambda, \bar u, \bar m)\in \mathcal{S}$ and define $\bar{m}^{\epsilon}:=\xi^{\epsilon} \star \bar{m}$ where $\xi^{\epsilon}(x)=\frac{1}{\epsilon^{n}}\xi(\frac{x}{\epsilon})$. 
Then 
\begin{align*}
& \int_{0}^{T} \int_{\overline{B}_{R}}\Big(F(x, m^{T}(t))- F(x, \bar m)\Big)d(m^{T}(t)-\bar m)\ dt\\
=& \int_{0}^{T} \int_{\overline{B}_{R}}{\Big(F(x, m^{T}(t))- F(x, \bar m)\Big)d(m^{T}(t)-\bar m^{\epsilon})\ dxdt}\\
& + \int_{0}^{T} \int_{\overline{B}_{R}}{\Big(F(x, m^{T}(t))- F(x, \bar m)\Big)d(\bar m^{\epsilon}-\bar m)\ dt}.
\end{align*}
Since $\bar m^{\epsilon} \to \bar m$ as $\epsilon \to 0$, we have that there exists $\bar\epsilon>0$ such that, for every $\epsilon \leq \bar\epsilon$,
\begin{equation*}
\int_{0}^{T} \int_{\overline{B}_{R}}{\Big(F(x, m^{T}(t))- F(x, \bar m)\Big)d(\bar m^{\epsilon}-\bar m)\ dt} \leq 1.
\end{equation*}
On the other hand, by the convexity of $H$, we obtain  
\begin{align*}
&\int_{0}^{T} \int_{\overline{B}_{R}}{\Big(F(x, m^{T}(t))- F(x, \bar m)\Big)d(m^{T}(t)-\bar m^{\epsilon})\ dxdt}\\
\leq &\ \int_{0}^{T} \int_{\overline{B}_{R}}{\Big(H(x,D\bar u)-H(x,Du^{T})-\Bigl\langle D_{p}H(x,Du^{T}), D(\bar u - u^{T}) \Bigl\rangle \Big)\ m^{T}(t)\ dxdt} \\
& + \int_{0}^{T}\int_{\overline{B}_{R}}{\Big( H(x,Du^{T})-H(x, D\bar u) - \Bigl\langle D_{p}H(x, D\bar u), D(u^{T}- \bar u) \Bigl\rangle \Big)\ \bar{m}^{\epsilon}\ dxdt} \\
& + \int_{0}^{T} \int_{\overline{B}_{R}}{\left(F(x, m^{T}(t))-F(x, \bar m)\right)\ (m^{T}(t)-\bar{m}^{\epsilon})\ dxdt}. 
\end{align*}
Recombining the terms on right hand-side of the above expression, we get  
\begin{align}\label{vv}
\begin{split}
&\int_{0}^{T} \int_{\overline{B}_{R}}{\Big(F(x, m^{T}(t))- F(x, \bar m)\Big)d(m^{T}(t)-\bar m^{\epsilon})\ dxdt}\\
\leq &\ \underbrace{\int_{0}^{T}\int_{\overline{B}_{R}}{\left( H(x, D\bar u(x)) - F(x, \bar m)\right)\ (m^{T}(t)-\bar m^{\epsilon})\ dxdt}}_{A}\\ 
&\ +  \underbrace{\int_{0}^{T}\int_{\overline{B}_{R}}{\Bigl\langle D\left(u^{T}(t,x)-\bar u(x)\right), D_{p}H(x, D\bar u(x)) \Bigl\rangle\ \bar m^{\epsilon}\ dxdt}}_{B} \\
&\ +\underbrace{\int_{0}^{T}\int_{\overline{B}_{R}}{\Bigl\langle D\left(u^{T}(t,x)-\bar u(x)\right), D_{p}H(x, D\bar u^{T}(t,x)) \Bigl\rangle\ \bar m^{T}(t)\ dxdt}}_{C}\\
&\ + \underbrace{\int_{0}^{T}\int_{\overline{B}_{R}}{-\left(H(x, Du^{T}(t,x))-F(x, m^{T}(t)) \right)\ (m^{T}(t)-\bar m^{\epsilon})\ dxdt}}_{D}.
\end{split}
\end{align}

In the following, we analyze each term on the right hand-side of \eqref{vv}.
Since $m^{T}(t), \bar{m}^{\epsilon}$ are probability measures and $B_{R_1}$, $K_{0} \subset \overline{B}_{R}$ for every $t \in [0,T]$, we have that
\begin{align}\label{a1}
	A=\int_{0}^{T} \int_{\overline{B}_{R}}{(H(x,D\bar u(x))-F(x,\bar m))\ (m^{T}(t)-\bar{m}^{\epsilon})\ dxdt}=0.
\end{align}
In order to study term $B$, define $$V^{\epsilon}:=\frac{\xi^{\epsilon} \star (\bar m D_{p}H(\cdot, D\bar u))}{\bar{m}^{\epsilon}}.$$ Then, we have that $-\text{div}(\bar{m}^{\epsilon}V^{\epsilon})=0$ in $\mathbb{R}^{n}$. We multiply this equality by $u^{T}(t,x)-\bar u(x)$ and integrate on $(0,T) \times \overline{B}_{R}$, to obtain
\begin{align}\label{dd}
\begin{split}
0 =& \int_{0}^{T} \int_{\overline{B}_{R}}{ \Bigl\langle D(u^{T}(t,x)-\bar u(x)), V^{\epsilon}(x) \Bigl\rangle\ \bar{m}^{\epsilon}\ dxdt}\\
&-\int_{0}^{T}\int_{\partial\overline{B}_{R}} {\left(u^{T}(t,x)-\bar u(x)\right)\left\langle \hat\nu, V^{\epsilon}(x) \right\rangle\ \bar m^{\epsilon}\ dSdt},
\end{split}
\end{align}
where $\hat\nu(x)=\frac{x}{R}$ is the outward unit normal  to $\partial\overline{B}_{R}$. 
Since $\partial \overline{B}_{R} \cap \supp(\bar m) = \emptyset$, we get
\begin{equation*}
\int_{0}^{T}\int_{\partial\overline{B}_{R}} {\left(u^{T}(t,x)-\bar u(x)\right)\left\langle \nu, V^{\epsilon}(x) \right\rangle\ \bar m^{\epsilon}\ dxdt}=0.
\end{equation*}
Thus, \eqref{dd} can be rewritten as 
\begin{align}\label{rr}
0 &= \int_{0}^{T} \int_{\overline{B}_{R}}{ \Bigl\langle D(u^{T}(t,x)-\bar u(x)), D_{p}H(x,D\bar u(x)) \Bigl\rangle\ \bar{m}^{\epsilon}\ dxdt}+R_{\epsilon},
\end{align}
where 
\begin{align*} 
& R_{\epsilon}=\\
 &\ \int_{0}^{T}\int_{\overline{B}_{R}}\int_{\mathbb{R}^{n}}\xi^{\epsilon}(x-y)\Bigl\langle D\left(u^{T}(t,x)-\bar u(x)\right), D_{p}H(y, D\bar u (y))-D_{p}H(x, D\bar u(x))\Bigl\rangle\ \bar m(dy)dxdt.
\end{align*} 
\noindent
By the definition of $R_{\epsilon}$, we have that

\begin{align*}
|R_{\epsilon}| \leq & T \sup_{t \in [0,T]}\| D\left(u^{T}(t, \cdot)-\bar u(\cdot)\right)\|_{\infty, \overline{B}_{R}}\cdot\\
& \int_{\overline{B}_{R}}\int_{\mathbb{R}^{n}}\xi^{\epsilon}(x-y) \big|D_{p}H(y, D\bar u (y))-D_{p}H(x, D\bar u(x))\big| \bar m(dy)dx.
\end{align*}
We now prove that $R_{\epsilon}\to 0$ as $\epsilon \to 0$. First, observe that
the integral term on the right-hand side of the above inequality can be rewritten as follows
\begin{align*}
&\int_{\overline{B}_{R}}\int_{\mathbb{R}^{n}}\xi^{\epsilon}(x-y) \big|D_{p}H(y, D\bar u (y))-D_{p}H(x, D\bar u(x))\big|\ \bar m(dy)dx\\
=&\ \int_{\mathbb{R}^{n}}\int_{\overline{B}_{R}}\xi^{\epsilon}(x-y) \big|D_{p}H(y, D\bar u (y))-D_{p}H(x, D\bar u(x))\big|\ dx\bar m(dy)\\
\leq &\ \int_{\mathbb{R}^{n}}\int_{\mathbb{R}^{n}}\xi^{\epsilon}(x-y) \big|D_{p}H(y, D\bar u (y))-D_{p}H(x, D\bar u(x))\big|\ dx\bar m(dy)\\
= &\ \int_{\mathbb{R}^{n}}\int_{\mathbb{R}^{n}}\xi(z) \big|D_{p}H(y, D\bar u (y))-D_{p}H(y+\epsilon z, D\bar u(y+\epsilon z))\big|\ dz\bar m(dy).
\end{align*}
Since $D_{p}H(\cdot, D\bar u(\cdot))$ is bounded and $\xi$ has compact support, for some $C>0$ we have that  
\begin{align*}
\int_{\mathbb{R}^{n}}\xi(z) \big|D_{p}H(y, D\bar u (y))-D_{p}H(y+\epsilon z, D\bar u(y+\epsilon z))\big|\ dz \leq C.
\end{align*}
Moreover, by the continuity of $D\bar u$ on $\supp(\bar m)$, we deduce that for any $y \in \supp(\bar m)$
\begin{align*}
\int_{\mathbb{R}^{n}}\xi(z) \big|D_{p}H(y, D\bar u (y))-D_{p}H(y+\epsilon z, D\bar u(y+\epsilon z))\big|\ dz \to 0 \quad \text{as}\ \epsilon \to 0.
\end{align*}
Therefore, by Lebesgue's dominated convergence theorem, we get  
\begin{align*}
\int_{\mathbb{R}^{n}}\int_{\mathbb{R}^{n}}\xi(z) \big|D_{p}H(y, D\bar u (y))-D_{p}H(y+\epsilon z, D\bar u(y+\epsilon z))\big|\ dz\bar m(dy) \to 0 \quad \text{as}\ \epsilon \to 0.
\end{align*}
In conclusion, we have that
$B \leq CT|R_{\epsilon}|
$.
In particular, for any $T>0$ there exists $\epsilon_{T}>0$ such that 
$
B \leq 1
$ for all $\epsilon \leq \epsilon_{T}$.
Hereafter, for any $T>0$ fix $\epsilon=\epsilon_{T}\wedge\bar\epsilon$.

Finally, we  to bound $C+D$. 
By the continuity equation in system \eqref{lab1}, we deduce that $$\partial_{t}(m^{T}(t)-\bar{m}^{\epsilon})-\text{div}\Big(m^{T}(t)D_{p}H(x, Du^{T}(t,x))\Big)=0,$$ in the sense of distributions.
Multiplying this equality by $u^{T}-\bar u$ and integrating in space-time, we conclude that
\begin{align*}\label{a2}
0 = & \int_{\overline{B}_{R}}{\Big((u^{f}(x)-\bar u(x))(m^{T}(T)-\bar m^{\epsilon})-(u^{T}(0,x)-\bar u(x))(m_{0}-\bar m^{\epsilon})\Big)\ dx}\\
&+ \int_{0}^{T} \int_{\overline{B}_{R}}{-\Big(H(x,Du^{T}(t,x))-F(x,m^{T})\Big)(m^{T}(t)-\bar{m}^{\epsilon})\ dxdt} \\
& + \int_{0}^{T} \int_{\overline{B}_{R}}{\Bigl\langle D(u^{T}(t,x)-\bar u(x)), D_{p}H(x, Du^{T}(t,x))\Bigl\rangle\ m^{T}(t)dxdt}\\
&-\int_{0}^{T}\int_{\partial \overline{B}_{R}}{\left(u^{T}(t,x)-\bar u(x)\right)\left\langle \hat\nu, D_{p}H(x, D u^{T}(t,x)) \right\rangle\  m^{T}(t)\ dSdt},
\end{align*}
where $\hat\nu(x)=\frac{x}{R}$ is the outward unit normal  to $\partial\overline{B}_{R}$.
Again, since $\partial \overline{B}_{R} \cap \overline{B}_{R_1} = \emptyset$, the integral over $\partial \overline{B}_{R}$ is zero.
In addition, the first integral is uniformly bounded  with respect to $T$, because
\begin{align*}
	\left| \int_{\overline{B}_{R}}{(u^{f}(x)-\bar u(x))(m^{T}(T)-\bar{m}^{\epsilon})}\ dx\right| \leq 2 ( \|u^{f}\|_{\infty,\overline{B}_{R}} + \|\bar u\|_{\infty,\overline{B}_{R}}).
\end{align*}
Since $m_{0}$,  $\bar{m}^{\epsilon}$ are probability measures, by Poincar\'e's inequality we deduce that
\begin{align*}
& \left| \int_{\overline{B}_{R}}{(u^{T}(0,x)-\bar u(x))(m_{0}-\bar{m}^{\epsilon})\ dx} \right|\\
 \leq & \left| \int_{\overline{B}_{R}}{\left( u^{T}(0,x)- \frac{1}{\mathcal{L}^{n}(\overline{B}_{R})}\int_{\overline{B}_{R}}{u^{T}(0,y)\ dy}\right)\ (m_{0}-\bar{m}^{\epsilon}) dx } \right| 
 + \left| \int_{\overline{B}_{R}}{\bar u(x)(m_{0}-\bar{m}^{\epsilon})\ dx} \right|\\
 \leq & \ N(R)(\|Du^{T}(0,\cdot)\|_{\infty,\overline{B}_{R}} + \|\bar u\|_{\infty,\overline{B}_{R}}).
\end{align*}
for some constant $N(R)>0$. Therefore,  
\begin{equation*}
C+D \leq 2 \big( \|u^{f}\|_{\infty,\overline{B}_{R}} + \|\bar u\|_{\infty,\overline{B}_{R}}\big) +  N(R)\big(\|Du^{T}(0,\cdot)\|_{\infty,\overline{B}_{R}} + \|\bar u\|_{\infty,\overline{B}_{R}}\big).
\end{equation*}

In view of the above estimates on $A$, $B$, $C$, and $D$, we conclude that 
\begin{align*}
\int_{0}^{T} \int_{\overline{B}_{R}}{(m^{T}(t)-\bar{m}^{\epsilon})\left(F(x, m^{T}(t))-F(x, \bar m)\right)\ dxdt} \leq {C}(R),
\end{align*}
for some constant ${C}(R) >0 $.
\end{proof}

\begin{theorem}[Convergence of solutions of \eqref{lab1}]\label{MR2}
Let $(\bar\lambda, \bar u, \bar m)\in\mathcal{S}$. Let $R_1$ be as in Proposition \ref{mea}. Then for any $R>R_1$, there exists a constant $C(R)>0$ such that for every $T\geq 1$ the solution $(u^{T}, m^{T})$ of system \eqref{lab1} satisfies
        \begin{equation}\label{lab38}
	\sup_{t \in [0,T]} \Big\|\frac{u^{T}(t, \cdot)-\bar u (\cdot)}{T} + \bar\lambda\left(1-\frac{t}{T}\right) \Big\|_{\infty, \overline{B}_{R}} \leq \frac{C(R)}{T^{\frac{1}{n+2}}}, 
	\end{equation}
	\begin{equation}\label{lab39}
	\frac{1}{T}\int_{0}^{T}{\big\| F(\cdot, m^{T}(s))- F(\cdot, \bar m) \big\|_{\infty, \overline{B}_{R}} ds} \leq \frac{C(R)}{T^{\frac{1}{n+2}}}.
	\end{equation}
	\end{theorem}

\begin{proof}
	Fix a radius $R>R_{1}$.
	Define
	\begin{equation*}
	w(t,x):=\bar u(x)-\bar\lambda (T-t), \quad \forall (x,t)\in \mathbb{R}^n\times[0,T].
	\end{equation*}
	Since $(\bar\lambda, \bar u, \bar m)$ is a solution of \eqref{lab2},
we have that $w$ is a viscosity solution of the following Cauchy problem
	\begin{align*}\label{cal}
	\begin{split}
	&\begin{cases}
	-\partial_{t} w + H(x, Dw)=F(x, \bar m) \quad \text{in} \quad (0,T)\times \mathbb{R}^{n},  \\ w(T,x)=\bar u (x) \quad \quad\quad\quad\quad\quad\,\,\,\,\,\,\,\,\text{in} \quad \mathbb{R}^{n}.
	\end{cases}
	\end{split}
	\end{align*}
	So, $w(t,x)$ can be represented as the value function of the following minimization problem
	\begin{equation}\label{ww}
	w(t,x)= \inf_{\gamma \in \Gamma_{t,T}(x)} \left\{\int_{t}^{T}{L_{\bar m}\left(\gamma(s), \dot\gamma(s)\right)\ ds} + \bar u(\gamma(T))\right\}, \quad \forall (x,t) \in \mathbb{R}^{n} \times [0,T].
	\end{equation}
Since $(u^{T}, m^{T})$ is a solution of \eqref{lab1}, then we get 	that
	\begin{align*}
	\begin{cases}
	-\partial_{t} u^{T} + H(x, Du^{T})=F(x, m^{T}) \quad \text{in} \quad (0,T) \times \mathbb{R}^{n},  \\ u^{T}(T,x)=u^{f}(x) \quad \quad\quad\quad\quad\quad\quad\,\,\,\,  \text{in} \quad  \mathbb{R}^{n}.
	\end{cases}
	\end{align*}

We prove  inequality \eqref{lab39} first. For any given $(x,t)\in \overline{B}_{R}\times [0,T]$, let $\gamma^{*}:[t,T]\to \mathbb{R}^n$ be a minimizer of  problem \eqref{ww}.   By Lemma \ref{leA} below and H\"{o}lder's inequality, we get
	\begin{align*}
	 \begin{split}
	 & \int_{t}^{T}{\| F(\cdot, m^{T}(s)) -F(\cdot, \bar m)\| _{\infty, \overline{B}_{R}} {\bf 1}_ {\overline{B}_{R}}(\gamma^{*}(s))\ \frac{ds}{T}} \\
	 \leq  & \ C((\|DF\|_\infty) \int_{t}^{T}{\| F(\cdot, m^{T}(s))- F(\cdot, \bar m) \|_{2, \overline{B}_{R}}^{\frac{2}{n+2}} {\bf 1}_{\overline{B}_{R}}(\gamma^{*}(s)) \frac{ds}{T}} \\
	  \leq & \ \frac{C(\|DF\|_\infty)}{T} \left( \int_{t}^{T}{\| F(\cdot, m^{T}(s)) -F(\cdot, \bar m) \|_{2,\overline{B}_{R}}^{2}\ ds} \right) ^{\frac{1}{n+2}} \left(\int_{t}^{T}{ {\bf 1}_{\overline{B}_{R}}(\gamma^{*}(s))\ ds } \right) ^{\frac{n+1}{n+2}} .
	  \end{split}
	  \end{align*} 
	  Now, by assumption {\bf (F3)} and Lemma \ref{energy} the term
	  $$
	   \left( \int_{t}^{T}{\| F(\cdot, m^{T}(s)) -F(\cdot, \bar m) \|_{2,\overline{B}_{R}}^{2}\ ds} \right) ^{\frac{1}{n+2}}
	   $$
	  is bounded by a constant depending only on $R$, while
	  $$
	  \left(\int_{t}^{T}{ {\bf 1}_{\overline{B}_{R}}(\gamma^{*}(s))\ ds } \right) ^{\frac{n+1}{n+2}}\leq T^{\frac{n+1}{n+2}}.
	  $$
	  Inequality \eqref{lab39} follows.
	  
	   Next, we  prove \eqref{lab38}. By the definition of $w$, we have that
	\begin{align}\label{lab43}
	\begin{split}
	& u^{T}(t,x)-w(t,x)
	\\ \leq &  \ \int_{t}^{T} {L_{m^{T}(s)}(\gamma^{*}(s), \dot\gamma^{*}(s))\ ds }+ u^{f}(\gamma^{*}(T)) - \int_{t}^{T} {L_{\bar m}(\gamma^{*}(s), \dot\gamma^{*}(s))\ ds} - \bar u(\gamma^{*}(T)) 
	\\=& \ u^{f}(\gamma^{*}(T)) - \bar u(\gamma^{*}(T)) + \int_{t}^{T}{ \left( F(\gamma^{*}(s), m^{T}(s)) - F(\gamma^{*}(s), \bar m) \right)\ ds}.
	\end{split}
	\end{align}
	By  \eqref{lab43}, we get 
	\begin{align*} 
	 \frac{ u^{T}(t,x) - w(t,x)}{T} \leq   \underbrace{\bigl|\frac{u^{f}(\gamma^{*}(T))-\bar u (\gamma^{*}(T))}{T}\bigl|}_{A} + \underbrace{\frac{1}{T}\int_{t}^{T}{\bigl|F(\gamma^{*}(s), m^{T}(s)) -F(\gamma^{*}(s), \bar m)\bigl| ds}}_{B}. 
	 \end{align*} 
	 Let us first consider term $B$. Note that \begin{align*} 
	& \frac{1}{T}\int_{t}^{T}{\bigl|F(\gamma^{*}(s), m^{T}(s))-F(\gamma^{*}(s), \bar m)\bigl|\ ds}\\ 
	 = &  \int_{t}^{T}{\bigl|F(\gamma^{*}(s), m^{T}(s)) -F(\gamma^{*}(s), \bar m)\bigl| {\bf 1}_{\overline{B}_{R}}(\gamma^{*}(s))\ \frac{ds}{T}} \\
	 & + \ \int_{t}^{T}{\bigl|F(\gamma^{*}(s), m^{T}(s))-F(\gamma^{*}(s), \bar m)\bigl|{\bf 1}_{\mathbb{R}^{n}\backslash \overline{B}_{R}}(\gamma^{*}(s))\ \frac{ds}{T}}. 
	 \end{align*}
	Since $F$ is bounded, by Theorem \ref{4-3} we know that the second integral on the right hand-side of the above equality goes to zero as $T\to \infty$. As for the first integral, observe that	 
	\begin{align*} 
	& \int_{t}^{T}{\bigl|F(\gamma^{*}(s), m^{T}(s)) -F(\gamma^{*}(s), \bar m)\bigl| {\bf 1}_{\overline{B}_{R}}(\gamma^{*}(s))\ \frac{ds}{T}} \\
	 \leq & \int_{t}^{T}{\| F(\cdot, m^{T}(s))- F(\cdot, \bar m)\|_{\infty, \overline{B}_{R}} {\bf 1}_{\overline{B}_{R}}(\gamma^{*}(s))\ \frac{ds}{T}}.
	 \end{align*} 
	Thus, we  estimate the term on the right-hand side of the above iniequality as we did above.

	  Now, we give a bound for $A$.  Since $\bar u$ and $u^f$ are Lipschitz continuous,  we deduce that both functions grow at most linearly at infinity, i.e., there exists a constant $N >0$ such that
	$$
	 \bar u(y), u^{f}(y) \leq\ N(1+ |y|), \quad \forall y \in \mathbb{R}^{n}.
	$$	  
	Also, by Corollary \ref{vel},  $\sup_{s\in[t,T]}|\dot\gamma^{*}(s)| \leq \chi'(R)$ for some
	constant $\chi'(R)>0$. 	  
	  Therefore, by Theorem \ref{4-3} there exists a positive constant $M_{R}$ such that for any $T>0$  we have that 
	 $$
	 |\bar u(\gamma^{*}(T))|, |u^{f}(\gamma^{*}(T))| \leq\ N(1+|\gamma^*(T)|) \leq\ N\Big(1+ (R+ \chi^{\prime}(R) M_{R})\Big).
	 $$
	   Thus, we conclude that $A \leq O(\frac{1}{T})$.
	So, combining \eqref{lab43} and \eqref{lab39}  we obtain 
	  \[
	  \frac{ u^{T}(t,x) - w(t,x)}{T} \leq\ 
	  \frac{C(R)}{T^{\frac{1}{n+2}}}.
	  \]
	  Moreover, for any given $(x,t)\in \overline{B}_{R}\times [0,T]$, let $\xi^{\ast}(\cdot)$ be a minimizer of  problem \eqref{infimum}. Since 
	$u^T$ is the value function of  \eqref{infimum}, we have that
	\begin{align}\label{lab42} 
	\begin{split}
	&w(t,x)-u^{T}(t,x)
	\\ \leq & \ \int_{t}^{T} {L_{\bar m}(\xi^{*}(s), \dot\xi^{*}(s))\ ds }+ \bar u(\xi^{*}(T)) - \int_{t}^{T} {L_{m^{T}(s)}(\xi^{*}(s), \dot\xi^{*}(s))\ ds} - u^{f}(\xi^{*}(T)) 
	\\=& \ \bar{u}(\xi^{*}(T)) - u^{f}(\xi^{*}(T)) + \int_{t}^{T}{ \left( F(\xi^{*}(s), \bar m) - F(\xi^{*}(s), m^{T}(s)) \right)\ ds}.
	\end{split}
	\end{align}
So, by almost the same arguments used above, one obtains
\[
	  \frac{w(t,x)- u^{T}(t,x)}{T} \leq\ 
	  \frac{C(R)}{T^{\frac{1}{n+2}}},
	  \]
which concludes the proof of  \eqref{lab38}.\end{proof}
\begin{remark}\em
	In view of Remark \ref{re} and the above proof, it is clear that the Theorem \ref{MR2} still holds true if assumption {\bf (F5)} is replaced by assumption {\bf (F5')}.
\end{remark}

\medskip

\appendix

\section{Appendix}

In this section, we first give the proof of ($ii$) of Theorem \ref{MR1} and then show  Lemma \ref{leA}, which was used in the proof of Theorem \ref{MR2}.

\subsection{Proof of ($ii$) of Theorem \ref{MR1} }

Let $(c(H_{\bar m_{1}}), \bar u_{1}, \bar m_{1})$, $(c(H_{\bar m_{2}}), \bar u_{2}, \bar m_{2})\in\mathcal{S}$, where $H_{\bar{m}_i}$ is defined in \eqref{hm} and $c(H_{\bar{m}_i})$ denotes the Ma\~n\'e critical value of $H_{\bar{m}_i}$, $i=1,2$. Let $\epsilon >0$, $\xi : \mathbb{R}^{n} \to \mathbb{R}$ be a smooth, nonnegative, symmetric kernel with a support contained in the unit ball and of integral one. Set $\xi^{\epsilon}(x)=\frac{1}{\epsilon^{n}}\xi(\frac{x}{\epsilon})$ for $i=1,2$. Define $m_{i}^{\epsilon}:=\xi^{\epsilon} \star \bar m_{i}$, i.e., the convolution of $\xi^{\epsilon}$ and $\bar m_{i}$, and
 
$$
V_{i}^{\epsilon}(x):=\frac{\xi^{\epsilon}(x) \star \big(\bar m_{i} D_{p}H(x, D\bar u_{i}(x))\big)}{m_{i}^{\epsilon}}.
$$ 
It is clear that $-\text{div}(m_{i}^{\epsilon}V_{i}^{\epsilon})=0$ in $\mathbb{R}^{n}$. We multiply this equality by $\bar u_{1} - \bar u_{2}$, integrate by parts and subtract the resulting formulas to get $$ \int_{\mathbb{R}^{n}}{\Bigl\langle D(\bar u_{1} - \bar u_{2}), m_{1}^{\epsilon}V_{1}^{\epsilon} - m_{2}^{\epsilon}V_{2}^{\epsilon}\Bigl\rangle\ dx} =0.$$ 
 Hence 
 \begin{align}\label{lab100}
 \begin{split}
 	0&=\int_{\mathbb{R}^{n}}{\Bigl\langle D(\bar u_{1} - \bar u_{2}), \xi^{\epsilon} \star (\bar m_{1}D_{p}H(\cdot, D\bar u_{1})-\bar m_{2}D_{p}H(\cdot, D\bar u_{2}))\Bigl\rangle \ dx}\\
& =\int_{\mathbb{R}^{n}}{\Bigl\langle D(\bar u_{1}-\bar u_{2}), m_{1}^{\epsilon}D_{p}H(x, D\bar u_{1})-m_{2}^{\epsilon}D_{p}H(x, D\bar u_{2}) \Bigl\rangle\ dx} + R_{\epsilon},
 \end{split}
 \end{align}
where we have defined 
\begin{align*} 
R_{\epsilon}:= &\ \int_{\mathbb{R}^{n}}\int_{\mathbb{R}^{n}}\xi^{\epsilon}(x-y)\Bigl\langle D(\bar u_{1}-\bar u_{2})(x), D_{p}H(y, D\bar u_{1}(y))-D_{p}H(x, D\bar u_{1}(x))\Bigl\rangle\ \bar m_{1}(dy)dx\\
&\ - \int_{\mathbb{R}^{n}}\int_{\mathbb{R}^{n}}\xi^{\epsilon}(x-y) \Bigl\langle D(\bar u_{1}-\bar u_{2})(x), D_{p}H(y, D\bar u_{2}(y))-D_{p}H(x, D\bar u_{2}(x)) \Bigl\rangle\ \bar m_{2}(dy)dx .
\end{align*} 
In particular, with almost the same considerations as in Lemma \ref{energy}, one can prove that $R_{\epsilon} \to 0$ as $\epsilon \to 0$.

Next, for $i=1,2$, we multiply $H_{\bar m_{i}}(x, D\bar u_{i})=c(H_{\bar m_{i}})$ by $(m_{1}^{\epsilon}- m_{2}^{\epsilon})$, integrate in space and subtract the second identity from the first one to get  $$\int_{\mathbb{R}^{n}}{\Big(H(x, D\bar u_1)-H(x, D\bar u_2)-F(x, \bar m_1) +F(x, \bar m_2)\Big)\ (m_{1}^{\epsilon}-m_{2}^{\epsilon})dx}=0.
$$ 
Now, combining the above equality with \eqref{lab100} we obtain 
\begin{align}\label{L}
\begin{split} 
-R_{\epsilon}
&= \ \int_{\mathbb{R}^{n}}{\Bigl\langle D(\bar u_{1}-\bar u_{2}), m_{1}^{\epsilon}D_{p}H(x, D\bar u_{1})-m_{2}^{\epsilon}D_{p}H(x, D\bar u_{2})\Bigl\rangle\ dx} \\ &- \int_{\mathbb{R}^{n}}\Big(H(x, D\bar u_{1})-H(x, D\bar u_{2})-F(x, \bar m_{1})+F(x, \bar m_{2})\Big)\ (m_{1}^{\epsilon}-m_{2}^{\epsilon})dx.  
\end{split}
\end{align} 
Hereafter, we denote $H(x, D\bar u_{i})$ by $H_{i}$  and $D_{p}H(x, D\bar u_{i})$ by $D_{p}H_{i}$ for $i=1,2$. Then, following Lasry-Lions \cite{bib:LL3}, by \eqref{L} $R_{\epsilon}$ can be recast as 
\begin{align*}
-R_{\epsilon} =& \int_{\mathbb{R}^{n}}\Big(H_{2}-H_{1}-\Bigl\langle DH_{1}, D(\bar u_{2} -\bar u_{1}) \Bigl\rangle\Big)\ m_{1}^{\epsilon}dx \\
&\ + \int_{\mathbb{R}^{n}}\Big(H_1-H_{2}-\Bigl\langle DH_{2}, D(\bar u_{1}-\bar u_{2})\Bigl\rangle\Big)\ m_{2}^{\epsilon}dx \\ 
&\ +\int_{\mathbb{R}^{n}}\Big(F(x, \bar m_1)-F(x,\bar  m_2)\Big)\ (m_{1}^{\epsilon}-m_{2}^{\epsilon})dx. 
\end{align*} 
Owing to the convexity of $H$ with respect to the second argument, the terms $H_2-H_1-\langle DH_{1}, D(\bar u_{2}-\bar u_{1}) \rangle $ and $H_1-H_2-\langle DH_{2}, D(\bar u_{1}-\bar u_{2}) \rangle $ are nonnegative. So we get 
 $$ 
 \int_{\mathbb{R}^{n}}\Big(F(x, \bar m_1)-F(x, \bar m_2)\Big)\ (m_{1}^{\epsilon}-m_{2}^{\epsilon})dx \leq -R_{\epsilon}.
 $$ 
 Letting $\epsilon \to 0$ we have that 
 $$ 
 \int_{\mathbb{R}^{n}}{\Big(F(x, \bar m_{1})- F(x, \bar m_{2})\Big)d(\bar m_{1}-\bar m_{2})} \leq 0,$$ which together with assumption {\bf (F3)} implies  that $$ \int_{\mathbb{R}^{n}}{\Big(F(x, \bar m_{1})-F(x, \bar m_{2})\Big)^{2} dx} \leq 0.
 $$
Hence, for every $x \in \mathbb{R}^{n}$, we have that $F(x,\bar m_{1})=F(x,\bar m_{2})$ and then, since $H_{\bar m_{1}}=H_{\bar m_{2}}$, we deduce that $c(H_{\bar m_1})=c(H_{\bar m_2})$.
\begin{equation*}
\eqno{\square}
\end{equation*}


\subsection{A technical lemma}
\begin{lemma}\label{leA} 
	For any $n \geq1$ and $D \geq 0$ there exists a constant $c=c(n,D)$ such that any Lipschitz continuous function $f: \mathbb{R}^{n} \to \mathbb{R}$, with ${\rm{Lip}}(f) \leq D$, satisfies 
\begin{equation}\label{lab17}
	 \|f\|_{\infty} \leq c(n,D) \|f\|_{2}^{\frac{2}{n+2}}.
\end{equation}
\end{lemma}

\begin{proof}
Observe, first, that \eqref{lab17} is trivial if $f \not\in L^{2}(\mathbb{R}^{n})$.

For any $\delta >0$ and any $x^0 \in \mathbb{R}^n$, we set
    \[
     Q(x^{0},\delta)=\{ x \in \mathbb{R}^{n}: \max_{1\leq i\leq n}|x_{i}-x_{i}^{0}| \leq \delta \}.
     \]
     Then, for any such cube $Q(x^{0},\delta)$ , we have that
     $$
   |f(x)| \leq |f(y)| + D\delta \sqrt{n}, \quad  \forall x,y \in Q(x^{0},\delta).
   $$
	Taking the square of both sides and integrating over $Q(x^{0},\delta)$, we obtain
	 $$ \int_{Q(x^{0},\delta)} {|f(x)|^{2} dy} \leq 2\int_{Q(x^{0},\delta)}{\Big( |f(y)|^{2} + D^{2}\delta^{2}n \Big) dy}, $$ 
	 which implies $$2^{n}\delta^{n} |f(x)|^{2} \leq 2 \int_{Q(x^{0},\delta)}{|f(y)|^{2} dy} + 2^{n+1}D^{2}n\delta^{n+2}.$$ Taking the supremum over the cube $Q$ we have that $$\frac{1}{2} \sup_{x \in Q(x^{0},\delta)} |f(x)|^{2} \leq \frac{1}{2^n \delta^{n}} \int_{Q(x^{0},\delta)}{|f(y)|^{2}dy} +2^n D^{2}n\delta^{2} \leq \frac{1}{2^n \delta^{n}} \int_{\mathbb{R}^{n}}{ |f(y)|^{2}dy} +2^n nD^{2}\delta^{2}.$$ 
	 Since $x^{0}$ may be taken arbitrarily in $\mathbb{R}^{n}$, we deduce that  $$ \frac{1}{2} \sup_{x \in \mathbb{R}^{n}}|f(x)|^{2} \leq \frac{1}{2^n \delta^{n}} \int_{\mathbb{R}^{n}}{|f(y)|^{2}dy} +nD^{2}\delta^{2}=:G(\delta).$$ 
	 Taking the minimum of $G(\delta)$ for $\delta \in (0, +\infty)$, yields the conclusion.
	 \end{proof}
	 
The following example shows that the exponent in \eqref{lab17} is optimal, in the sense that such a bound would be false for any exponent $\theta > \frac{2}{n+2}$.

\begin{example}
Consider a function on $\mathbb{R}$ defined by
\begin{align*}
f(t):=
	&\begin{cases}
	\frac{1}{k}-t,\quad t \in [0, \frac{1}{k}],\\ 
	0,\quad \text{otherwise},
	\end{cases}
	\end{align*}
for some $k\in\mathbb{N}$. Then, we have
\begin{align*}
\|f\|_{\infty} =\frac{1}{k},
\quad  \|f\|_{2}=\left(\int_{0}^{\frac{1}{k}}{(\frac{1}{k}-t)^{2}dt}\right)^{\frac{1}{2}}=\left(\frac{1}{3}\right)^{\frac{1}{2}}\left(\frac{1}{k}\right)^{\frac{3}{2}}.
\end{align*}
Since $n=1$, then 
 \[
 \|f\|^{\frac{2}{n+2}}_{2}=\left(\frac{1}{3}\right)^{\frac{1}{3}}\frac{1}{k},
 \]
which implies that the estimate \eqref{lab17} is optimal.
\end{example}


\medskip
\medskip

\noindent {\bf Acknowledgements:}
Piermarco Cannarsa was partly supported by Istituto Nazionate di Alta Matematica (GNAMPA 2018 Research Projects) and by the MIUR Excellence Department Project awarded to the Department of Mathematics, University of Rome Tor Vergata, CUP E83C18000100006. 
Wei Cheng was partly supported by Natural Scientific Foundation of China (Grant No. 11871267, No. 11631006 and No.11790272). Kaizhi Wang was partly supported by National Natural Science Foundation of China (Grant No. 11771283) and by China Scholarship Council (Grant No. 201706235019).



\begin{thebibliography}{abc}
      %
	\bibitem{bib:AGS} L. Ambrosio, N. Gigli and G. Savare, Gradient Flows in Metric Spaces and in the Space of Probability Measures. Second edition, Lectures in Mathematics ETH Z\"urich. Birkh\"auser Verlag, Basel, 2008.
       %
       \bibitem{bib:AFS} J. Aubin, H. Frankowska, Set-valued Analysis, Birkha\"user, Boston, 1990.
       
  
	\bibitem{bib:BCD} M. Bardi, I. Capuzzo Dolcetta, Optimal Control and Viscosity Solutions of Hamilton-Jacobi Equations, Birkh\"auser, Boston, 1997.	
	%
	\bibitem{bib:BB} P. Billingsley,  Convergence of Probability Measures, John Wiley \& Sons, Inc., New York, 1999.
	%
	\bibitem{bib:BK} K. Border, Fixed Point Theorems with Applications to Economics and Game Theory, Cambridge University Press, Cambridge, 1989.
	%
	\bibitem{bib:SC} P. Cannarsa, C. Sinestrari, Semiconcave Functions, Hamilton-Jacobi Equations, and Optimal Control, Birkh\"auser, Boston, 2004.
	
	\bibitem{bib:C} P. Cannarsa, T. D'Aprile, 
Introduction to Measure Theory and Functional Analysis. Translated from the 2008 Italian original. Unitext, 89. La Matematica per il 3+2. Springer, Cham, 2015. 
	
	\bibitem{bib:CAR} P. Cardaliaguet,  Long time average of first order mean field games and weak KAM theory, Dyn. Games Appl. {\bf 3} (2013),  473--488.
	%
	\bibitem{bib:CN} P. Cardaliaguet, Notes on Mean Field Games. \url{http://www.ceremade.dauphine.fr/~cardaliaguet/MFG20130420.pdf}
	
	
	
	\bibitem{bib:CIPP} G. Contreras, R. Iturriaga, G. Paternain and M. Paternain, Lagrangian graphs, minimizing measures and Ma\~n\'e critical values, Geom. Funct. Anal. {\bf 8} (1998), 788--809.	
	
	
	%
	\bibitem{bib:GC} G. Contreras, Action potential and weak KAM solutions,  Calc. Var. Partial Differential Equations, {\bf 13} (2001), 427--458.
	%
	
	\bibitem{bib:GP} G. Da Prato, An Introduction to Infinite-dimensional Analysis. Revised and extended from the 2001 original by Da Prato. Universitext. Springer-Verlag, Berlin, 2006. 
	
	
	
	%
	\bibitem{bib:AF} A. Fathi, Th\'eor\`eme KAM faible et th\'eorie de Mather sur les syst\`emes lagrangiens. (French) [A weak KAM theorem and Mather's theory of Lagrangian systems] C. R. Acad. Sci. Paris S\'er. I Math. {\bf 324} (1997), 1043--1046. 	
	
	
	\bibitem{bib:FA} A. Fathi, Weak KAM Theorem and Lagrangian Dynamics. 
	\url{http://www.math.u-bordeaux.fr/~pthieull/Recherche/KamFaible/Publications/Fathi2008_01.pdf}
	
	
	%
	\bibitem{bib:FM} A. Fathi, E. Maderna, Weak KAM theorem on non compact manifolds, NoDEA Nonlinear Differential Equations Appl. {\bf 14} (2007), 1--27.
	
	
		%
	
	\bibitem{bib:HCM2} M. Huang, R. P. Malham\'e  and P. E. Caines, Large population stochastic dynamic games: closed-loop McKean-Vlasov systems and the Nash certainty equivalence principle, Commun. Inf.  Syst. {\bf 6} (2006),  221--251.
	%
	
	
	\bibitem{bib:HCM1} M. Huang, P. E. Caines,  and R. P. Malham\'e, Large-population cost-coupled LQG problems with nonuniform agents: Individual-mass behavior and decentralized $\epsilon$-Nash equilibria, Automatic Control, IEEE Trans. Automat. Control {\bf 52}  (2007),  1560--1571.



	\bibitem{bib:LL1} J.-M. Lasry, P.-L. Lions, Jeux \`a champ moyen. I. Le cas stationnaire. (French) [Mean field games. I. The stationary case] C. R. Math. Acad. Sci. Paris {\bf 343} (2006), 619--625. 
	%
	\bibitem{bib:LL2} J.-M. Lasry,  P.-L. Lions, Jeux \`a champ moyen. II. Horizon fini et controle optimal. (French) [Mean field games. II. Finite horizon and optimal control] C. R. Math. Acad. Sci. Paris {\bf 343} (2006),  679--684. 	
	
	
	%
	\bibitem{bib:LL3}J.-M. Lasry,  P.-L. Lions, Mean field games, Jpn. J. Math. {\bf 2} (2007), 229--260.
	%
	
	\bibitem{bib:LPV} P.-L. Lions, G. Papanicolaou and S. R. Varadhan, Homogenization of Hamilton-Jacobi equations. \url{http://localwww.math.unipd.it/~bardi/didattica/Nonlinear_PDE_%20homogenization_Dott_%202011/LPV87.pdf}
	%
	\bibitem{bib:Mat}
	 J. Mather, Action minimizing invariant measures for positive definite Lagrangian systems, Math. Z. {\bf 207} (1991), 169--207. 
	
	\bibitem{bib:CV} C. Villani, Topics in optimal transportation. 
Graduate Studies in Mathematics, {\bf 58}. American Mathematical Society, Providence, RI, 2003. 
%

	
\end{thebibliography}
\end{document}